\documentclass[a4paper,11pt]{article}
\usepackage[a4paper]{geometry}
\usepackage[latin1]{inputenc}
\usepackage[T1]{fontenc}

\usepackage{lipsum} %Pour faire des essais

\usepackage[francais,english]{babel}
\usepackage{lmodern}
\rmfamily
\DeclareFontShape{T1}{lmr}{b}{sc}{<->ssub*cmr/bx/sc}{}
\DeclareFontShape{T1}{lmr}{bx}{sc}{<->ssub*cmr/bx/sc}{}

\usepackage{amstext,amsthm,amssymb,amscd,version,a4}
\usepackage[mathscr]{eucal}
\usepackage{mathrsfs}
\usepackage{epsf}
  \usepackage[all]{xy}
\usepackage{amsmath}
\def\C{\mathbb{C}}
\def\N{\mathbb{N}}

\def\R{\mathbb{R}}
\def\Z{\mathbb{Z}}

\def\kk{\mathfrak{k}}
\def\pp{\mathfrak{p}}

\def\hh{\mathfrak{h}}

\def\uu{\mathfrak{u}}
\renewcommand\gg{\mathfrak{g}}

\def\Kk{\mathcal{K}}
\def\Bb{\mathcal{B}}
\def\Ee{\mathcal{E}}
\def\Ff{\mathcal{F}}
\def\Ww{\mathcal{W}}
\def\Tt{\mathcal{T}}
\def\Vv{\mathcal{V}}

\def\Hha{\mathcal{H}}

\newcommand{\Lsg}[1][]{\Lambda_\sigma^{#1} \gg}

\def\LsG{\Lambda_\sigma G}
\def\LG{\Lambda G}
\def\Lg{\Lambda \gg}
\def\LK{\Lambda K}
\def\Lk{\Lambda \kk}
\def\Gg{\mathcal{G}}
\def\Mm{\mathcal{M}}
\def\Nn{\mathcal{N}}
\def\Ff{\mathcal{F}}
\def\Kk{\mathcal{K}}

\newtheorem{thmintro}{Theorem}

\newtheorem{thm}{Theorem}[section]

\newtheorem{lem}[thm]{Lemma}
\newtheorem{prop}[thm]{Proposition}
\newtheorem{conj}[thm]{Conjecture}
\newtheorem{cor}[thm]{Corollary}

\theoremstyle{definition}

\newtheorem{rem}[thm]{Remark}

\newtheorem{defn}[thm]{Definition}

\newtheorem{qtn}[thm]{Question}

\newcommand{\comm}[1]{}

\DeclareMathOperator{\Hom}{Hom}
\DeclareMathOperator{\sh}{sh}
\DeclareMathOperator{\Sh}{Sh}
\DeclareMathOperator{\Tr}{Tr}

\DeclareMathOperator{\Ad}{Ad}
\DeclareMathOperator{\End}{End}
\DeclareMathOperator{\Aut}{Aut}

\DeclareMathOperator{\Gr}{Gr}
\DeclareMathOperator{\Id}{Id}

\DeclareMathOperator{\tot}{tot}

\renewcommand{\Im}{\operatorname{Im}}

\def\Oo{\mathcal{O}}
\def\Aa{\mathcal{A}}

\def\CP{\mathbb{CP}}
\def\Dd{\mathcal{D}}
\def\Tt{\mathcal{T}}
\def\Ee{\mathcal{E}}
\def\Ddc{\mathcal{\check{D}}}

\usepackage{hyperref}
\usepackage[nottoc]{tocbibind}

\title{Loop Hodge structures and harmonic bundles}
\author{Jeremy DANIEL}

\def\d{\partial}

\begin{document}

\maketitle

\begin{abstract} We define the notion of a loop Hodge structure -- an infinite dimensional generalization of a Hodge structure -- and prove that a suitable variation of this object over a complex manifold is equivalent to the datum of a harmonic bundle. Hence one can study harmonic bundles using classical tools of Hodge theory, especially the existence of a period map (with values in an infinite dimensional period domain). Among other applications, we prove an integrality result for the Hitchin energy class of a harmonic bundle.
\end{abstract}

\section*{Introduction}

A loop Hodge structure is an infinite dimensional vector space with some extra structure that incorporates features from both Hodge theory and loop group theory. In this paper, we define these objects and prove that a harmonic bundle over a complex manifold can be thought as a variation of loop Hodge structures. In particular we associate a period map to any harmonic bundle, with values in an infinite dimensional period domain. 

\subsection{Non-abelian Hodge theory}

Let $X$ be a connected compact K\"ahler manifold. The classical Hodge theory of $X$ studies the interaction between topological and holomorphic cohomological invariants of $X$. It starts with the Hodge decomposition theorem (\cite{GrHar}, p.116) which asserts that the Betti cohomology of $X$ can be written as a direct sum of its Dolbeault cohomology spaces. 

Non-abelian Hodge theory extends this correspondence to non-abelian invariants. On the topological side, one considers the complex linear representations of the fundamental group $\pi_1(X)$; by the Riemann-Hilbert correspondence, they can be thought as complex flat vector bundles $(\Ee,D)$ over $X$. On the holomorphic side, one is interested in Higgs bundles $(\Ee,\theta)$, where $\Ee$ is a holomorphic vector bundle on $X$ and $\theta$ is a holomorphic $1$-form with values in $\End(\Ee)$, satisfying $\theta \wedge \theta = 0$. The fundamental theorem in non-abelian Hodge theory asserts that there is a 1--1 correspondence between semisimple representations of $\pi_1(X)$ and polystable Higgs bundles with vanishing Chern classes. We refer to \cite{Sim_higloc} or \cite{LePot} for more details.

This correspondence is established \emph{via} the notion of harmonic bundle, which is a flat bundle endowed with a particular Hermitian metric, see definition \ref{def_harm}. Any harmonic bundle gives naturally rise to both a flat bundle and a Higgs bundle. Conversely, theorems of Corlette \cite{Cor} and Donaldson \cite{Don} on the one hand, and of Hitchin \cite{Hitch} and Simpson \cite{Sim_hig} on the other hand, characterize respectively the flat and Higgs bundles that appear in this way. This suggests a deeper study of harmonic bundles.\\

Variations of (complex polarized) Hodge structures (cf. definition \ref{VHS}) provide the simplest examples of harmonic bundles. Let $(\Ee,D,h, \Ee = \oplus_p \Ee^p)$ be such a variation. Then the Hermitian inner product $\hat{h}$ defined by changing the sign of the Hermitian form $h$ on the bundles $\Ee^p$, for odd $p$, is a harmonic metric on $(\Ee,D)$. 

The flat bundle associated with the variation of Hodge structures is simply $(\Ee,D)$. In order to define the Higgs bundle associated with it, we recall that the Hodge (decreasing) filtration $F^p \Ee := \oplus_{q \geq p} \Ee^q$ is holomorphic and satisfies the so-called Griffiths transversality condition with respect to $D$. Then, the Higgs bundle is $(\Gr_{F^\bullet}\Ee, \Gr_{F^\bullet} D)$: since the filtration $F^\bullet$ is holomorphic, the graded space carries a holomorphic structure and the transversality condition implies that $D$ induces an operator on the graded space, which is in fact a holomorphic $1$-form with values in $\End(\Gr_{F^\bullet} \Ee)$. We refer to \cite{Sim_higloc}, p.44, for more details.\\

Harmonic bundles can thus be considered as generalizations of variations of Hodge structures. This is quite surprising for two reasons. On the one hand, a harmonic bundle is not really defined in a variational way: the relation between the flat connection and the metric is quite intricate and does not look like the holomorphicity and transversality conditions encountered in classical Hodge theory. On the other hand, the geometry of a variation of Hodge structures is governed by a complex homogeneous space (the period domain), whereas it is the symmetric space of Hermitian metrics on $\C^n$ that appears for harmonic bundles. These two kind of homogenous spaces are very different in nature: the symmetric space is not even a complex manifold.

\begin{qtn} \label{qtn_intro}
 Is there a natural way to understand harmonic bundles as variations of some structure ?
\end{qtn}

One fruitful approach has been to consider harmonic bundles as variations of twistor structures \cite{Sim_twis}. However, as noticed by Mochizuki in \cite{Mo1}, 1.3.2, there is still no analogue for the period domain in this context.

\subsection{Summary of our results}

In section \ref{VLHS}, we define the fundamental object of this paper: a \emph{loop Hodge structure}. Let us give now an overview of it; precise definitions are in subsection \ref{LHS}. 

A loop Hodge structure is a Hilbert space $K$ endowed with some additional structures, among which the most important one is a non-degenerate and indefinite Hermitian form, as for (classical) Hodge structures. A subspace of $K$ is also part of the datum of a loop Hodge structure; it is the analogue of the Hodge filtration in a classical Hodge structure. Any loop Hodge structure $K$ is naturally isomorphic to $L^2(S^1,V)$, where $V$ is some finite-dimensional complex vector space, hence as a space of \emph{loops}, explaining the terminology and the relation with loop groups.

We then define variations of loop Hodge structures (subsection \ref{VLHS_harm}) and prove:

\begin{thmintro}\label{thm_introA}
 Let $X$ be a complex manifold. The category of variation of loop Hodge structures on $X$ is equivalent to the category of harmonic bundles on $X$.
\end{thmintro}

The set of variations of loop Hodge structures over a complex manifold $X$ is naturally endowed with an action of $S^1$. Then:

\begin{thmintro}\label{thm_introB}
 The variations of classical Hodge structures can be identified with the variations of loop Hodge structures which are isomorphic to every variation of loop Hodge structures in their $S^1$-orbit.
\end{thmintro}

In section \ref{sec_per_dom} we construct a classifying space for loop Hodge structures of fixed type: the \emph{period domain} $\Dd$. Its geometry is similar to the one of classical period domains but it is infinite dimensional: it is a complex manifold,  homogeneous for the action of some real form of the loop group $$\Lambda GL(n,\C) = \{\gamma: S^1 \rightarrow GL(n,\C), \text{ with some regularity}\}.$$ Moreover the tangent bundle of $\Dd$ carries a natural (finite-dimensional) horizontal distribution. As in classical Hodge theory we prove:

\begin{thmintro}\label{thm_introC}
 The datum of a harmonic bundle over $X$ is equivalent to the datum of a period map: a holomorphic and horizontal map from $\tilde{X}$ to $\Dd$, that obeys a certain equivariance condition.
\end{thmintro}

These results give a satisfactory answer to question \ref{qtn_intro}. In this paper we derive two consequences of this approach.\\

Let $(\Ee,\theta)$ be a Higgs bundle over a complex manifold $X$, which comes from a harmonic bundle. The vanishing of the Higgs field $\theta$ defines a holomorphic distribution on $X$. The involutivity of this distribution was previously obtained in \cite{Mok} and we call this distribution the \emph{Higgs foliation} on $X$. The following theorem is proved in subsection \ref{hig_fol}:

\begin{thmintro}\label{thm_introD}
 The Higgs foliation is integrable.
\end{thmintro}

Involutivity is indeed a weaker notion than integrability since the Higgs foliation is singular in general.\\

Finally, in section \ref{det_line}, we define a certain holomorphic line bundle on the period domain $\Dd$ and use it to deduce some properties of harmonic bundles over a complex manifold $X$. If $(\Ee,D,h)$ is a harmonic bundle over $X$, we write $\theta$ for the Higgs field it induces and $\theta^\ast$ for its adjoint with respect to the harmonic metric. Then, the $(1,1)$-form $\beta_X = \frac {1}{4i\pi} \Tr(\theta \wedge \theta^\ast)$ is closed and defines a class in $H^2(X,\R)$. This energy class was already considered in \cite{Hitch}; in subsection \ref{hitch_energy}, we prove:

\begin{thmintro}\label{thm_introE}
 The pullback of the class of $\beta_X$ to the universal cover of $X$ lives in $H^2(\tilde{X},\Z)$.
\end{thmintro}

We then use this result to give another proof of a known case of the Carlson-Toledo conjecture.

\subsection{Further remarks}

\paragraph{Nilpotent and $SL(2)$-orbit theorems}

In the study of classical Hodge structures, the period domain has been used by Schmid to obtain the nilpotent and $SL(2)$-orbit theorems \cite{Schm}. It is thus natural to ask for an analogue of these theorems for general harmonic bundles. As remarked by Mochizuki in the introduction of \cite{Mo1}, some new phenomena appear when one considers general harmonic bundles; for instance the Higgs field $\theta$ is not nilpotent in general but it is for variations of Hodge structures. 

Restricting to the class of tame nilpotent harmonic bundles with trivial parabolic structure over the puncture disk $\Delta^\ast$, the author has proved an analogue of the nilpotent orbit theorem in \cite{my_thesis}; moreover an analogue of the $SL(2)$-theorem is stated there without proof. It is not clear to the author whether this approach will give new results on the asymptotics of harmonic bundles; hence we prefer not to publish this work yet.

\paragraph{The infinite dimensional period domain}

The infinite dimensional period domain $\Dd$ already appeared in the literature (in particular in \cite{DPW} and \cite{EschTrib}). In these references, it is shown that any pluriharmonic map from a (simply-connected) manifold to some symmetric space can be lifted to a holomorphic and horizontal map to the period domain. The precise relation between their construction and ours is explored in \cite{my_thesis}. From this point of view, our major constribution is the intrinsic definition of this period domain as a classifying space for loop Hodge structures.

\subsection{Acknowledgments}

This work is based on the author's Ph.D. thesis \cite{my_thesis}, done at Universit\'e Paris Diderot. I am very grateful to my supervisor Bruno Klingler, who asked me to think about question \ref{qtn_intro}. 

I also thank Olivier Biquard, Yohan Brunebarbe, Philippe Eyssidieux, Takuro Mochizuki, Claude Sabbah and Carlos Simpson: I have greatly benefited from mathematical discussions with them.

\newpage

\tableofcontents
\newpage

\section{Variations of loop Hodge structures} \label{VLHS}

In this section, we define the notion of loop Hodge structures and we prove that their variations, in a suitable meaning, correspond to harmonic bundles. The relation with classical Hodge structures is then explained in detail. Our general motivation is to extend in this way some results on variations of Hodge structures to harmonic bundles. This is very similar to the \emph{meta-theorem} of Simpson \cite{Sim_twis}; the major difference will appear in next section, with the construction of a period domain.

\subsection{Loop Hodge structures} \label{LHS}

We introduce the fundamental notion of loop Hodge structure. We will need to consider infinite dimensional complex vector spaces endowed with a nondegenerate indefinite metric. Such spaces can behave badly but Krein spaces share a lot of properties with Hilbert spaces. The standard reference for this topic is \cite{Bog}.

\subsubsection{Krein spaces}

Let $K$ be a (possibly infinite-dimensional) complex vector space and let $B$ be a Hermitian form on $K$; we do not assume that $B$ is positive definite. A subspace $W$ of $K$ is \emph{positive} (resp. \emph{negative}) if $B_{|W \otimes \bar{W}}$ is positive definite (resp. negative definite). Such a subspace is endowed with a pre-Hilbert structure; we say that it is \emph{intrinsically complete} if this structure is Hilbert, that is if the induced distance is complete.

We say that $K$ is a \emph{Krein space} if there exists a $B$-orthogonal decomposition $K = K_+ \oplus^\perp K_-$ such that $K_+$ (resp. $K_-$) is an intrinsically complete positive (resp. negative) subspace. Such a decomposition is called a \emph{fundamental decomposition} of the Krein space $K$. This is not unique, even for finite-dimensional spaces endowed with an indefinite metric.

From a fundamental decomposition of the Krein space $K$, one can define a unique Hilbert inner product $h$ on $K$ such that $K_-$ and $K_+$ are orthogonal for $h$, $h = B$ on $H_+$ and $h = -B$ on $H_-$. This inner product depends on the fundamental decomposition but we have the following proposition:

\begin{prop}{(Corollary IV.6.3 in \cite{Bog})}
Let $K$ be a Krein space. The Hilbert structures coming from two fundamental decompositions of $K$ define the same topology on $K$.
\end{prop}

Hence, one can speak of \emph{the Hilbert topology} of the Krein space $K$. Any reference to the topology of $K$ will be with respect to this one.//

Let $W$ be a non-degenerate closed subspace of $K$. Then $W$ admits a decomposition $W = W_+ \oplus^\perp W_-$, with $W_+$ (resp. $W_-$) a positive (resp. negative) closed subspace of $K$ (theorem V.3.1 in \cite{Bog}). However, $W_+$ and $W_-$ are not necessarily intrisically complete, that is $W$ is not necessarily a Krein space itself. The following proposition gives a characterization of closed subspaces of a Krein space, that are Krein spaces. We call them \emph{Krein subspaces}, though it does not depend on the ambient space.

\begin{prop}{(Theorems V.3.4 and V.3.5 in \cite{Bog})}\label{ortho}
Let $W$ be a closed subspace in a Krein space $K$. Let $W = W_+ \oplus^\perp W_-$ be a decomposition of $W$ as above. Then $W$ is a Krein subspace of $K$ if and only if there exists a fundamental decomposition $K = K_+ \oplus^\perp K_-$ such that $W_+ \subset K_+$ and $W_- \subset K_-$. Moreover, if $W$ is a Krein subspace of $K$, the orthogonal $W^\perp$ of $W$ in $K$ is a Krein subspace and satisfies $W \oplus^\perp W^\perp = K$.
\end{prop}

\subsubsection{Loop Hodge structures}

Let $K$ be a Krein space and let $T$ be an anti-isometric operator on $K$, that is $B(Tu,Tv) = -B(u,v)$ for $u$ and $v$ in $K$.

\begin{defn}\label{sub_out}
A closed subspace $W$ of $K$ is an \emph{outgoing subspace} for $T$ if
\begin{itemize}
 \item $W$ is a Krein subspace of $K$;
 \item $T W \subset W$;
 \item $\bigcap_{n \in \N} T^n W = 0$;
 \item $\bigcup_{n \in \Z} T^n W$ is dense in $K$;
 \item the orthogonal of $TW$ in $W$ is positive definite.
\end{itemize}
\end{defn}

This terminology comes from \cite{LaPh} and \cite{PS}, where it is defined in a Hilbert setting; in this case, $T$ is an isometric operator. In a Krein space $K$, an \emph{outgoing operator} $T$ is an anti-isometric operator that admits an outgoing subspace. An \emph{outgoing Krein space} $(K,B,T)$ is a Krein space $(K,B)$ endowed with an outgoing operator $T$. It is easy to see that a non-trivial outgoing Krein space is infinite-dimensional and that the positive and negative components in any fundamental decomposition are infinite-dimensional too. In particular, the Krein metric has to be indefinite. We will sometimes write $(K,T)$ or $K$ for an outgoing Krein space.

\begin{defn}
 A \emph{loop Hodge structure} $(K,B,T,W)$ is the datum of an outgoing Krein space $(K,B,T)$ and an outgoing subspace $W$ for $T$.

An \emph{isomorphism} from a loop Hodge structure $(K,B,T,W)$ to a loop Hodge structure $(K',B',T',W')$ is a Krein isometry $\phi: (K,B) \rightarrow (K',B')$ that intertwines the outgoing operators and such that $\phi(W) = W'$.
\end{defn}

This terminology is justified by the fact that classical Hodge structures can be realized as loop Hodge structures. This will be discussed in subsection \ref{class}.

\begin{defn}\label{Hod_filt}
The \emph{Hodge filtration} of a loop Hodge structure $(K,B,T,W)$ is the decreasing filtration $F^\bullet K$ of $K$, given by $F^i K := T^i W$. It is a complete topologically separated filtration, meaning that $\bigcap_{n \in \Z} F^n K = 0$ and $\bigcup_{n \in \Z} F^n K$ is dense in $K$.
\end{defn}

\subsubsection{Canonical form of a loop Hodge structure}

Let $H$ be a Hilbert space, with inner product denoted by $h$. Let $K$ be the space of square-integrable functions from the circle $S^1$ to $H$. We define a Krein structure on $K$ by
$$B(f,g) := \int_{S^1} h(f(\lambda),g(-\lambda))d\nu(\lambda),$$
where $\nu$ is the invariant volume form on $S^1$ with integral $1$.

Using Fourier series, we write $f(\lambda) = \sum_{n \in \Z} f_n \lambda^n$ and $g(\lambda) = \sum_{n \in \Z} g_n \lambda^n$, with $(f_n)_{n \in \Z}, (g_n)_{n \in \Z}$ in $l^2(H)$. Then $B(f,g) = \sum_{n \in \Z} (-1)^n h(f_n,g_n)$. A fundamental decomposition of $K$ is obtained by taking the odd and even components of a function in $K$; for this fundamental decomposition, the induced Hilbert structure is the usual Hilbert structure on $L^2(S^1,H)$.

The \emph{right-shift operator} of $K$ is the operator $T$, given by $(Tf)(\lambda) = \lambda f(\lambda)$. In terms of the Fourier series representation, $T$ is defined by $T((a_n)_{n \in \Z}) = (a_{n-1})_{n \in \Z}$. This is an anti-isometric operator and in fact an outgoing operator. Indeed, the subspace $L_+^2(S^1,H)$ of functions with Fourier series concentrated in nonnegative degrees is easily seen to be an outgoing subspace for $T$. This space will be called the \emph{Fourier-nonnegative subspace}. One should notice that functions in $L_+^2(S^1,H)$ are holomorphic functions on the open unit disk in $\C$ whose powers series representation $\sum_{n \in \N} a_n z^n$ at $0$ satisfy $\sum_{n \in \N} |a_n|^2 < \infty$.

This defines a canonical loop Hodge structure on $L^2(S^1,H)$. The next proposition is an adaptation with Krein spaces of a well known-fact in abstract scattering theory \cite{PS}, \cite{LaPh}, \cite{HSW}.

\begin{prop}\label{kr_out}
 Let $(K,B,T,W)$ be a loop Hodge structure and let $H$ be the orthocomplement of $TW$ in $W$. Then, $H$ is a Hilbert space and there exists a canonical isomorphism of loop Hodge structures from $K$ to $L^2(S^1,H)$.
\end{prop}

\begin{proof}
 Since $W$ is a Krein subspace of $K$ and $T$ is an anti-isometric operator, $TW$ is a Krein space too. In particular, one has an orthogonal direct sum $W = TW \oplus^\perp H$ by proposition \ref{ortho}. The subspaces $T^n H, n \in \Z$ are in orthogonal direct sum. Indeed, if $i < j$, then $T^i H$ is orthogonal to $T^{i+1} W$ and $T^j H \subset T^j W \subset T^{i+1} W$. By induction, one gets an orthogonal direct sum 
$$W = \big(\bigoplus_{0 \leq i\leq n}^\perp T^i H\big) \oplus^\perp  T^{n+1} W,$$ 
for all positive $n$. Since $\bigcap_{n\in\N} T^n W = 0$, this implies that $W$ is the Hilbert sum of the spaces $T^n H$, for $n \geq 0$. Moreover, $\bigcup_{n \in \Z} T^n W$ is dense in $K$; hence $K$ is the Hilbert sum of the $T^n H$, for $n \in \Z$.

 The Krein space $K$ can thus be identified with $L^2(S^1, H)$, with its standard Krein metric (the Krein metric $B$ is by assumption positive definite on $H$) since $T$ is an anti-isometric operator. By construction, this identification intertwines $T$ with the right-shift operator of $L^2(S^1,H)$ and sends $W$ to the Fourier-nonnegative subspace of $L^2(S^1, H)$.
\end{proof}

\begin{rem}\label{kr_rem}
With the notations of proposition \ref{kr_out}, the isomorphism $\Phi$ from $L^2(S^1,H)$ to $K$ is given by
$$
\Phi: \sum_{n \in \Z} a_n \lambda^n \mapsto \sum_{n \in \Z} T^n a_n.
$$
\end{rem}

\begin{prop}\label{isom_Hilb}
 Let $(K,B,T)$ be an outgoing Krein space and let $W_1$ and $W_2$ be two outgoing subspaces. We denote by $H_i$ the orthocomplement of $TW_i$ in $W_i$. 
 Then $H_1$ and $H_2$ are isometric Hilbert spaces.
\end{prop}

\begin{proof}
 By proposition \ref{kr_out}, the outgoing Krein space $(K,B,T)$ is isomorphic to $L^2(S^1,H_i)$ (as outgoing Krein space). In particular, there is an isomorphism 
 $$
 \phi: L^2(S^1,H_1) \cong L^2(S^1,H_2),
 $$
 that intertwines the outgoing operators. Writing $\kappa_i$ for the cardinal of a Hilbert basis of $H_i$ (which we recall determines $H_i$ up to isomorphism), we obtain that $\kappa_1 \times \aleph_0 = \kappa_2 \times \aleph_0$, where $\aleph_0$ is the cardinal of the set of natural numbers. This concludes the proof if both $H_i$ are not separable. 

If both $H_i$ are separable (in particular if they are finite-dimensional) but not isomorphic, we can assume that $H_2$ is of finite dimension $n_2$ and $H_1$ is of finite dimension $n_1 > n_2$ or has a countable but not finite Hilbert basis. In all cases, an injection of $\C^{n_2+1}$ in $H_1$ gives a continuous injection $L^2(S^1,\C^{n_2 + 1}) \hookrightarrow L^2(S^1,H_1)$. Post-composing by $\phi$, we get a injective continuous linear map $\tilde{\phi}: L^2(S^1,\C^{n_2 + 1}) \rightarrow L^2(S^1,H_2)$ that intertwines the outgoing operators. It is sufficient to show that such a map cannot exist to conclude the proof.

By proposition \ref{comm_shift2}, there exists $M$ in $L^\infty(S^1,\End(\C^{n_2+1},H_2))$ such that $\phi$ is the evaluation map induced by $M$. For almost every $\lambda$, the kernel of $M(\lambda)$ is not reduced to $0$ for dimensional reasons. By using Gauss elimination, we can define a non-zero measurable vector-valued function $v(\lambda)$ in $\C^{n_2+1}$ such that $v(\lambda)$ is in the kernel of $M(\lambda)$ for almost every $\lambda$. Moreover, up to a renormalization of $v$, we can assume that $v$ lives in $L^2(S^1,\C^{n_2+1})$.

Then, $v$ is not zero but $\phi(v)$ is zero, contradicting that $\phi$ is injective
\end{proof}

In particular, in a loop Hodge structure $(K,B,T,W)$, the dimension of $H$ is independent of the outgoing subspace $W$. This will be called the \emph{virtual dimension} of the outgoing Krein space $(K,B,T)$. We will only consider loop Hodge structures of finite virtual dimension.

\subsubsection{Families of loop Hodge structures}

We refer to \cite{Lan} for basic notions in differential geometry in infinite dimension. Let $\pi: \Kk \rightarrow X$ be a Hilbert bundle over a differentiable manifold $X$. Let $\Bb$ be a fibrewise Hermitian form on $\Kk$ and let $\Tt$ be a section of the Banach bundle $\End(\Kk) \rightarrow X$. We say that $(\Kk,\Bb,\Tt)$ is an \emph{outgoing Krein bundle} if for every $x$ in $X$, there exists a neighbourhood $U$ of $x$, a Hilbert space $K$, a trivialization $\pi^{-1}(U) \cong U \times K$ and an outgoing Krein structure on $K$ such that the restriction of the trivialization to each fiber is an isomorphism of outgoing Krein spaces. In particular, each fiber of $\Kk$ has a structure of outgoing Krein space and the intrinsic topology of this Krein space coincides with its topology as a fiber of the Hilbert bundle $\Kk$. %One can notice that, by propositions \ref{comm_shift} and \ref{B_isom}, the datum of an outgoing Krein bundle of virtual dimension $n$ is equivalent to the datum of a principal $\Lambda^\infty_\sigma GL_n(\C)$-bundle.

Let $\pi: (\Kk,\Bb,\Tt) \rightarrow X$ be an outgoing Krein bundle. A \emph{subbundle of outgoing subspaces} $\Ww$ in $\Kk$ is a subset of $\Kk$ such that, for any $x$ in $X$, there exists a neighbourhood $U$ of $x$, an outgoing Krein space $K$ and a trivialization $\pi^{-1}(U) \cong U \times K$ of outgoing Krein bundles such that, in this trivialization, $\Kk \cap \pi^{-1}(U)$ is identified with $U \times W$, where $W$ is a fixed outgoing subspace of $K$.

\begin{defn}
 A \emph{family of loop Hodge structures} over a differentiable manifold $X$ is the datum of an outgoing Krein bundle (of finite virtual dimension) with a subbundle of outgoing subspaces. An \emph{isomorphism} of such families is a bundle map which is an isomorphism of loop Hodge structures on each fiber.
\end{defn}

Let $(\Ee,h)$ be a finite-dimensional Hermitian bundle over $X$. The union of the Hilbert spaces $L^2(S^1,E_x)$, $x \in X$, has a natural structure of Hilbert bundle. Moreover, this Hilbert bundle is in a natural way a family of loop Hodge structures, by the discussion before proposition \ref{kr_out}.

\begin{prop} \label{bund_grass}
Let $X$ be a differentiable manifold and let $(\Kk,\Bb,\Tt,\Ww)$ be a family of loop Hodge structures of virtual dimension $n$. Then this family is isomorphic to the family of loop Hodge structures on $L^2(S^1,\Ee)$, where $\Ee$ is the complex vector bundle of ranh $n$, defined to be the orthogonal complement of $T\Ww$ in $\Ww$.
\end{prop}

\begin{proof}
 Since in a trivialization, $\Ww$ becomes a fixed outgoing subspace in a fixed outgoing Krein space, the orthogonal complement $\Ee$ of $T\Ww$ in $\Ww$ is a smooth bundle of rank $n$ over $X$. Moreover, by the definition of outgoing subspaces, the Hermitian form $\Bb$ restricts to a Hermitian inner product on $\Ee$. Hence, one can indeed consider the outgoing Krein bundle $L^2(S^1,\Ee)$, and by proposition \ref{kr_out}, one gets a vector bundle map $\Kk \rightarrow L^2(S^1,\Ee)$ which is an isomorphism of loop Hodge structures in each fiber. Since locally in some trivialization, this isomorphism does not depend on the point in $X$, it is smooth. This concludes the proof.
\end{proof}

\begin{cor}
 Let $X$ be a differential manifold. The category of families of loop Hodge structures over $X$ (with morphisms being isomorphisms) is equivalent to the category of finite-dimensional Hermitian bundles over $X$ (with morphisms being bundle isometries).
\end{cor}

\subsection{Variations of loop Hodge structures and harmonic bundles} \label{VLHS_harm}

We recall the notion of \emph{harmonic bundle} on a complex manifold $X$, define an appropriate notion of variation of loop Hodge structures that mimics definition \ref{VHS} and explain the 1--1 correspondence between these two objects.

\subsubsection{Developing map of a metric}

Let $(X,x_0)$ be a pointed connected differentiable manifold. We choose a base-point $\tilde{x}_0$ in the fiber above $x_0$ of the universal cover $\pi: \tilde{X} \rightarrow X$. Let $(\Ee,D)$ be a flat bundle of rank $r$ over $X$, with a trivialization $E_{x_0} \cong \C^r$ of the fiber above $x_0$. The pullback bundle $\pi^\ast E$ can be trivialized to $\tilde{X} \times \C^r$ by parallel transport with respect to the flat connection $\pi^\ast D$ and base-point $\tilde{x}_0$. If $h$ is a Hermitian metric on $E$, the pullback metric $\pi^\ast h$ can thus be seen as a varying Hermitian inner product on the fixed space $\C^r$. The space of such inner products is homogeneous under the action of $GL(r,\C)$ by congruence, with stabilizer equal to the unitary group $U(r)$ at the standard inner product. We thus get a map $f: \tilde{X} \rightarrow GL(r,\C)/U(r)$, which is equivariant under the monodromy $\rho: \pi_1(X,x_0) \rightarrow GL(r,\C)$ of the flat bundle. This map is called the \emph{developing map} of the 
metric $h$. 
Conversely, such a map defines a Hermitian metric on the flat bundle $(E,D)$.

\subsubsection{Harmonic bundles} \label{Harm_bundles}

Let $(M, g^M)$ and $(N,g^N)$ be two Riemannian manifolds. If $f: M \rightarrow N$ is a smooth map, the bundle $W := T^\ast M \otimes f^{-1} TN$ on $M$ is endowed with a connection $\nabla^W$, induced by the Levi-Civita connections on $M$ and $N$.

\begin{defn}
 The \emph{tension field} $\tau(f)$ of a smooth map $f: M \rightarrow N$ is a tensor in $f^{-1} TN$, obtained by taking the trace of the tensor $\nabla^W df$, that lives in $(T^\ast M)^{\otimes 2} \otimes f^{-1} TN$ . We say that $f$ is \emph{harmonic} if its tension field vanishes.
\end{defn}

We refer to \cite{HeWo} for a variational point of view. In the particular case where $M$ is a complex manifold of dimension $1$, this notion of harmonic map is independent of the (necessarily K\"{a}hler) metric on $M$ (see \cite{HeWo}, example 2.2.12). Hence, the following definition makes sense:

\begin{defn}\label{def_plur}
 A smooth map $f: X \rightarrow N$ from a complex manifold $X$ to a Riemannian manifold $N$ is \emph{pluriharmonic} if its restriction to every complex curve of $X$ is harmonic.
\end{defn}

 It is known that a pluriharmonic map from a K\"ahler manifold $X$ to a Riemannian manifold $N$ is harmonic. The best result in the converse direction is the Siu-Sampson's theorem \cite{Siu},\cite{Samps}:

 \begin{thm}[Siu-Sampson]\label{Samp}
  Let $X$ be a compact K\"{a}hler manifold, let $N$ be a Riemannian manifold of non-positive Hermitian curvature and let $\rho: \pi_1(X) \rightarrow \text{Isom}(N)$ be a representation of $\pi_1(X)$ in the isometry group of $N$. Then, any $\rho$-equivariant harmonic map from $\tilde{X}$ to $N$ is pluriharmonic.
 \end{thm}

To be precise, this theorem follows from the results in the two quoted references. A self-contained proof of this theorem can be found in \cite{ABCKT}: it is given in a non-equivariant setting but it is asserted there that the proof works also with equivariant maps.\\

There is an equivalent notion of pluriharmonicity. Let $(\Ee,D)$ be a flat bundle over a complex manifold $X$ and let $h$ be a Hermitian metric on $\Ee$. There is a unique decomposition $D = \nabla + \alpha$, where $\nabla$ is a metric connection for $h$ and $\alpha$ is a $h$-Hermitian $1$-form with values in $\End(\Ee)$. We write $\nabla = \d + \bar{\d}$ and $\alpha = \theta + \theta^\ast$, for their decompositions in types. Since $\alpha$ is Hermitian, $\theta^\ast$ is the adjoint of $\theta$ with respect to $h$.

\begin{defn}\label{def_harm}
 The metric $h$ on the flat bundle $(\Ee,D)$ is \emph{harmonic} if the differential operator $\bar{\d} + \theta$ has vanishing square. We also say that $(\Ee,D,h)$ is a \emph{harmonic bundle}. 
\end{defn}

The vanishing of the square of $\bar{\d} + \theta$ is equivalent to the vanishing of $\bar{\d}^2$, $\bar{\d}\theta$ and $\theta \wedge \theta$. By the Koszul-Malgrange integrability theorem \cite{KoszMal}, $\bar{\d}$ gives a structure of holomorphic bundle on $\Ee$. For this holomorphic structure, $\theta$ is a holomorphic $1$-form satisfying the equation $\theta \wedge \theta = 0$. Such a datum $(\Ee,\bar{\d}, \theta)$ is called a \emph{Higgs bundle} and $\theta$ is the \emph{Higgs field}.

These two notions of harmonicity are related by the following proposition:

\begin{prop}
A metric $h$ on $(\Ee,D)$ is harmonic if and only if its developing map is pluriharmonic.
\end{prop}

\begin{proof}
We refer to proposition 1.1.17 in \cite{my_thesis}.
\end{proof}

\begin{rem}
 The terminology is a little ambiguous but well established and we will not use the more precise term of \emph{pluriharmonic bundle}. Anyway, in the case where $X$ is a compact K\"ahler manifold, harmonicity and pluriharmonicity are equivalent by theorem \ref{Samp} since the symmetric space $GL(r,\C)/U(r)$ has non-positive Hermitian curvature.
\end{rem}

\subsubsection{Circle of connections}

We consider a harmonic bundle $(\Ee,D,h)$ and write $D = \nabla + \alpha$, $\nabla = \d + \bar{\d}$ and $\alpha = \theta + \theta^\ast$ as before. For any $\lambda$ in $S^1$, we write $\alpha_\lambda$ for the Hermitian $1$-form $\lambda^{-1} \theta + \lambda \theta^\ast$ and $D_\lambda$ for the connection
$$ D_\lambda = \nabla + \alpha_\lambda.$$
The following lemma is well known.

\begin{lem}\label{Dlamb_harm}
The bundle $(\Ee,D_\lambda,h)$ is harmonic.
\end{lem}

\begin{proof}
We denote by $F_\nabla$ the curvature of the connection $\nabla$. The curvature of $D$ is $D^2 = F_\nabla + \nabla\alpha + \alpha \wedge \alpha$. Since $F_\nabla + \alpha \wedge \alpha$ is anti-Hermitian and $\nabla\alpha$ is Hermitian, the vanishing of $D^2$ is equivalent to $F_\nabla + \alpha \wedge \alpha = 0$ and $\nabla \alpha = 0$.

The curvature of $D_\lambda$ is $D_\lambda^2 = (F_\nabla + \alpha_\lambda \wedge \alpha_\lambda) + \nabla\alpha_\lambda$. The decomposition by types of the equality $\nabla \alpha = 0$ and the condition $\bar{\d}\theta = 0$ yield $\nabla \theta = \nabla \theta^\ast = 0$. Hence, $\nabla \alpha_\lambda = 0$. Moreover, $\alpha_\lambda \wedge \alpha_\lambda = \lambda^2 \theta^\ast \wedge \theta^\ast + (\theta \wedge \theta^\ast + \theta^\ast \wedge \theta) + \lambda^{-2} \theta \wedge \theta$. Since $\theta \wedge \theta = 0$, only the central term doesn't vanish. Hence, $F_\nabla + \alpha \wedge \alpha = F_\nabla + \alpha_\lambda \wedge \alpha_\lambda$ and the vanishing of $D^2$ is equivalent to the vanishing of $D_\lambda^2$.

It remains to show that the operator $\bar{\d} + \lambda^{-1} \theta$ has vanishing square. This is equivalent to $\lambda^{-1}\theta$ being a Higgs field for the $\bar{\d}$ holomorphic structure, which is obvious.
\end{proof}

The family $(D_\lambda)_{\lambda \in S^1}$ is called the \emph{circle of flat connections} of the harmonic bundle $(E,D,h)$. We give another lemma which relates the harmonicity of $E$ with the existence of such a circle of flat connections.

\begin{comment}
\begin{rem}\label{dec_conn}
 We will soon use that this circle of flat connections defines a connection on the Hilbert bundle $L^2(S^1,E)$. Indeed, the family depends smoothly on the parameter $\lambda$. By the Remark \ref{sm_fam}, this family of connections can be seen as a connection $\tilde{D}$ in the Hilbert bundle $L^2(S^1,E)$. This is defined as follows: if $f$ is a local section of $L^2(S^1,E)$, then $(\tilde{D}f)(\lambda,x) = (D_\lambda f(\lambda,\cdot))(x)$. Moreover, the flatness of the $D_\lambda$ implies the flatness of $\tilde{D}$.
\end{rem}
\end{comment}

\begin{lem}\label{crit_harm}
 Let $X$ be a complex manifold and let $(\Ee,h)$ be a Hermitian bundle over $X$. We assume that there exists a metric connection $\nabla$ and a $(1,0)$-form $\theta$ with values in $\End(\Ee)$ such that, for every $\lambda$ in $S^1$, the connection
 $$
 D_\lambda := \nabla + \lambda^{-1}\theta +  \lambda \theta^\ast
 $$
 is flat.

 Then $(\Ee,D_\lambda,h)$ is a harmonic bundle, for every $\lambda$ in $S^1$.
 \end{lem}

\begin{proof}
 By the previous lemma, it is sufficient to show that $(\Ee,D_1,h)$ is a harmonic bundle. The connection $D_1$ has vanishing curvature by assumption. We write $\nabla = \d + \bar{\d}$. The equality $D_\lambda^2 = 0$ gives
 $$
 0 = \lambda^{-2} \theta \wedge \theta + \lambda^{-1} \nabla \theta + (F_\nabla + \theta \wedge \theta^\ast + \theta^\ast \wedge \theta) + \lambda \nabla \theta^\ast + \lambda^2 \theta^\ast \wedge \theta^\ast.
 $$
 Since this is true for every $\lambda$, the terms with different exponents of $\lambda$ vanish separately. Hence, $\theta \wedge \theta = 0$ and $\nabla \theta = 0$, which implies that the $(1,1)$-part $\bar{\d}\theta = 0$ vanishes. Moreover, the constant term in $\lambda$ has $(0,2)$-type $\bar{\d}^2$, hence $\bar{\d}^2$ vanishes too.
\end{proof}

\subsubsection{Variations of loop Hodge structures}

The following definition mimics the definition \ref{VHS} of variations of Hodge structures.

\begin{defn}\label{VNCHS}
A \emph{variation of loop Hodge structures} over a complex manifold $X$ is the datum of a family of loop Hodge structures $(\Kk,\Bb,\Tt,\Ww)$ and a flat connection $D = \d + \bar{\d}$ on $\Kk$ such that:
\begin{enumerate}
   \item the Krein metric $\Bb$ and the outgoing operator $\Tt$ are $D$-flat;
   \item $\bar{\d}: \Ww \rightarrow \Ww \otimes \Aa^{0,1}$ \label{diff1};
   \item $D: \Ww \rightarrow \Tt^{-1}\Ww \otimes \Aa^1$ \label{diff2}.
\end{enumerate}
\end{defn}

\begin{rem}
 We recall that a family of loop Hodge structures $(\Kk,\Bb,\Tt,\Ww)$ carries a topological decreasing filtration given by $F^p \Kk := \Tt^p \Ww$. Since in a variation of loop Hodge structures $\Tt$ is $D$-flat, the two differential conditions given above are equivalent to the \emph{a priori} stronger ones:
\begin{enumerate}
   \item $\bar{\d}: F^p\Kk \rightarrow F^p\Kk \otimes \Aa^{0,1}$;
   \item $D: F^p\Kk \rightarrow F^{p-1}\Kk \otimes \Aa^1$.
\end{enumerate}
\end{rem}

\subsubsection{The equivalence of categories}\label{equiv_cat}

Let $X$ be a complex manifold. We write $\Hha_X$ for the category of harmonic bundles over $X$, the morphisms being the flat and isometric vector bundle maps. We write $\Vv_X$ for the category of variations of loop Hodge structures over $X$, the morphisms being the flat isomorphisms of family of Hodge structures. We define a functor $F: \Hha_X \rightarrow \Vv_X$ in the following way:

Let $(\Ee,D,h)$ be a harmonic bundle. Let $\nabla$ be an arbitrary connection on $\Ee$ and let $\tilde{\nabla}$ be its naive extension to $L^2(S^1,\Ee)$. The circle of flat connections $(D_\lambda)_{\lambda \in S^1}$ can be thought as an element in $\tilde{\nabla} + \Gamma(X,T^\ast X \otimes L^\infty(S^1,\End(\Ee)))$, hence as a connection $\tilde{D}$ on the Hilbert bundle $L^2(S^1,\End(\Ee))$ (see corollary \ref{conn_T}). The functor $F$ is defined by $F: (\Ee,D,h) \mapsto (L^2(S^1,\Ee),\tilde{D})$, where $L^2(S^1,\Ee)$ is endowed with its canonical structure of family of loop Hodge structures.

\begin{thm}\label{thm_fund}

The functor $F$ is well-defined on objects and establish an equivalence of categories from $\Hha_X$ to $\Vv_X$.
\end{thm}

\begin{proof}

We first prove that $F(\Ee,D,h)$ is indeed a variation of loop Hodge structures. Then, we construct a quasi-inverse $G: \Vv_X \rightarrow \Hha_X$. That $F$ and $G$ are quasi-inverse is straightforward and left to the reader, as the definition of $F$ and $G$ on morphisms.

Let $(\Ee,D,h)$ be a harmonic bundle and write 
$$(L^2(S^1,\Ee),\Bb,\Tt,L^2_+(S^1,\Ee),\tilde{D}) := F(\Ee,D,h).$$ 
We recall that $\tilde{D}$ is defined by $(\tilde{D}f)(\lambda,x) = D_\lambda f(\lambda,x)$, where $f$ is a local section of $L^2(S^1,\Ee)$ and $x$ is in $\Ee$. The connection $\tilde{D}$ is flat since all $D_\lambda$ are flat. By corollary \ref{conn_T}, $\Tt$ is $\tilde{D}$-flat. The Hermitian form $\Bb$ is $\tilde{D}$-flat too. Indeed, let $f$ and $g$ be local sections of $L^2(S^1,\End(\Ee))$. Then,
 \begin{align*}
 d\Bb(f,g) &= \int_{S^1} dh(f(\lambda),g(-\lambda)) d\nu(\lambda)\\
 &= \int_{S^1} \big( h(D_\lambda f(\lambda),g(-\lambda)) + h(f(\lambda),D_{-\lambda} g(-\lambda) \big) d\nu(\lambda)\\
 &= \int_{S^1} \big(h((\tilde{D}f)(\lambda),g(-\lambda)) + h(f(\lambda),(\tilde{D}g)(-\lambda))\big) d\nu(\lambda) \\
 &= \Bb(\tilde{D}f,g) + \Bb(f,\tilde{D}g).
 \end{align*}
For the second equality, we use that $D_\lambda$ and $D_{-\lambda}$ are adjoint for the metric $h$. Indeed, $D_\lambda = \nabla + \alpha_\lambda$ and $D_{-\lambda} = \nabla - \alpha_\lambda$, where $\nabla$ is a metric connection and $\alpha_\lambda$ is a Hermitian $1$-form. We emphasize that the flatness of the Hermitian form is the main motivation for considering indefinite Hermitian forms.

We now check the differential constraints \ref{diff1}. and \ref{diff2}. of definition \ref{VNCHS} on $L^2_+(S^1,\Ee)$. By definition, the Fourier series of elements in $L^2_+(S^1,\Ee)$ have only nonnegative powers of $\lambda$. On the other hand, the flat connection $\tilde{D}$ is equal to $\nabla + \lambda^{-1} \theta + \lambda \theta^\ast$, with the usual notations. Hence, its $(0,1)$-part is concentrated in nonnegative Fourier coefficients and only powers of $\lambda$ greater or equal to $-1$ appear in the Fourier series of $D$. This shows that the two differential constraints on $L^2_+(S^1,\Ee)$ are indeed satisfied.\\

Conversely, let $(\Kk,\Bb,\Tt,\Ww,D)$ be a variation of loop Hodge structures. By proposition \ref{bund_grass}, there is an isomorphism of families of loop Hodge structures $\Kk \cong L^2(S^1,\Ee)$, where $\Ee$ is the orthogonal of $\Tt\Ww$ in $\Ww$. Let $h$ denote the restriction of $\Bb$ to $\Ee \otimes \bar{\Ee}$; this is a Hermitian inner product.

Choose an arbitrary metric connection $\nabla$ on $(\Ee,h)$ and denote by $\tilde{\nabla}$ the induced naive connection on $L^2(S^1,\Ee)$. Since $\Tt$ is $D$-flat, by corollary \ref{conn_T}, $D$ can be written as
$$ D = \tilde{\nabla} + \omega,$$
where $\omega$ lives in $\Gamma(X,T^\ast X \otimes L^\infty(S^1,\End(\Ee))).$ Since $\nabla$ is metric for $h$, one checks easily that $\tilde{\nabla}$ is metric for $\Bb$. Moreover, $\Bb$ is $D$-flat; hence we get the equality:
$$ \Bb(\omega.f,g) + \Bb(f,\omega.g) = 0,$$
where $f$ and $g$ are local sections of $L^2(S^1,\Ee)$. Let $u$ and $v$ be local sections of $\Ee$, let $i$ be in $\Z$ and consider the local sections $f = \lambda^i u$, $g = v$. Then the previous equality gives:
\begin{align} \label{conn_met}
 h(\omega_{-i}.u,v) + (-1)^i h(u,\omega_i.v) = 0,
\end{align}
where $\omega = \sum_{n \in \Z} \omega_n \lambda^n$, with $\omega_j \in \Gamma(X,T^\ast X \otimes \End(\Ee))$. Considering the types of the $\omega_j$, this equality proves that the $(1,0)$-part of $\omega_j$ vanishes if, and only if the $(0,1)$-part of $\omega_{-j}$ vanishes.

The two differential constraints on $\Ww$ give the following information: $\omega_j$ vanishes for $j \leq -2$ (transversality condition) and $\omega_{-1}^{0,1}$ vanishes too. Hence, by equation (\ref{conn_met}), the only possible non-vanishing terms in $\omega = \sum_{n \in \Z} \omega_n \lambda^n$ are $\omega_0$, $\omega_{-1}^{1,0}$ and $\omega_1^{0,1}$. Thus, $D$ can be written as
$$ D = (\tilde{\nabla} + \omega_0) + \lambda^{-1}\omega_{-1} + \lambda \omega_1,$$
and $\omega_{-1}$ (resp. $\omega_1$) is of type $(1,0)$ (resp. of type $(0,1)$). The equations in (\ref{conn_met}) for $i=0$ and $i=1$ prove that $\omega_0$ is an anti-Hermitian form and that $\omega_1$ and $\omega_{-1}$ are adjoint. The connection $D$ comes from the circle of connections $(D_\lambda)_{\lambda \in S^1}$ defined, for every $\lambda$ in $S^1$, by $D_\lambda := (\tilde{\nabla} + \omega_0) + \lambda^{-1}\omega_{-1} + \lambda \omega_1$. Hence, all $D_\lambda$ are flat connections. We define the functor $G$ by
$$
G(\Kk,\Bb,\Tt,\Ww,D) = (\Ee,D_1,h)
$$ 
and we conclude by lemma \ref{crit_harm}.
\end{proof}

\subsection{Classical Hodge structures} \label{class}

In this subsection, we explain how to realize a variation of classical Hodge structures as a variation of loop Hodge structures, justifying the chosen terminology.

\subsubsection{Classical Hodge structures interpreted as loop Hodge structures}

We first recall the notion of classical Hodge structures and their variations:

\begin{defn}\label{HS}
A (complex polarized) \emph{Hodge structure} over $X$ is the datum of a finite-dimensional complex vector space $V$, a non-degenerate Hermitian form $h$ on $V$ and a $h$-orthogonal decomposition
$$ V = \oplus_{i \in \Z} V^i$$
such that $h$ is $(-1)^i$ positive definite on $i$.
\end{defn}

Thanks to the Hermitian form $h$, the datum of the decomposition is equivalent to the datum of the filtration $F^\bullet V$ defined by $F^i V := \oplus_{j \geq i} V^j$.% We now define the variations of Hodge structures; these are families of Hodge structures whose associated filtrations obey strong conditions.

\begin{defn}\label{VHS}
 A \emph{variation of Hodge structures} over $X$ is the datum of 
\begin{itemize}
 \item a flat bundle $(\Ee,D)$;
 \item a flat non-degenerate Hermitian form $h$;
 \item a smooth $h$-orthogonal decomposition $\Ee = \oplus_{p \in \Z} \Ee^p$
\end{itemize}
such that, writing $\Ff^p := \oplus_{q \geq p} \Ee^q$, one has
\begin{enumerate}
 \item each fiber is a Hodge structure \label{cond_ponct};
 \item $D^{0,1}: \Ff^p \rightarrow \Aa^{0,1} \otimes \Ff^p$ \label{cond_hol};
 \item $D: \Ff^p \rightarrow \Aa^1 \otimes \Ff^{p-1}$. \label{cond_trans}
\end{enumerate}
\end{defn}

Here $\Aa^1$ and $\Aa^{0,1}$ denote the sheaves of smooth $1$-forms and smooth $(0,1)$-forms and we identify the vector bundles $\Ff^p$ with their sheaves of smooth sections. Such variations were defined by Griffiths, as an abstraction of what happens in the Hodge decomposition theorem, when one deforms the complex structure of $X$ but preserves the underlying topological space. The condition \ref{cond_hol}. says that the filtration $\Ff^p$ varies holomorphically with respect to $D$, whereas the \emph{transversality} condition \ref{cond_trans}. says that this filtration is not far from being flat.

Let $(V,h,F^\bullet V)$ be a (complex polarized) Hodge structure. We write $V = \bigoplus_n V^n$ for the Hodge decomposition. Let $K$ be the Hilbert space $L^2(S^1,V)$ endowed with the Hermitian form
$$
B(f,g) = \int_{S^1} h(f(\lambda),g(-\lambda)) d\nu(\lambda).
$$
Let $T$ denote the right-shift operator. Let $W$ be the subspace of $K$ given by
$$
W = \widehat{\bigoplus}_{n \in \Z} T^{-n} F^n V.
$$

\begin{prop}\label{CHS_NCHS}
 The quadruple $(K,B,T,W)$ is a loop Hodge structure.
\end{prop}

\begin{proof}
 Since the Hermitian form $h$ is non-degenerate, the Hermitian form $B$ defines a structure of Krein space on $H$. A fundamental decomposition is for instance given by
$$K_+ := \big(\widehat{\bigoplus}_{n \text{\,even}} T^n V^{\text{\,even}}\big) \bigoplus \big(\widehat{\bigoplus}_{n \text{\,odd}} T^n V^{\text{\,odd}}\big)$$
and
$$K_- := \big(\widehat{\bigoplus}_{n \text{\,even}} T^n V^{\text{\,odd}}\big) \bigoplus \big(\widehat{\bigoplus}_{n \text{\,odd}} T^n V^{\text{\,even}}\big).$$
The operator $T$ is anti-isometric. Let us show that $W$ is an outgoing Krein subspace. All points in definition \ref{kr_out} are trivial, except that the orthogonal $E$ of $TW$ in $W$ is definite positive. This space is equal to
$$ E = \widehat{\bigoplus}_{n \in \Z} T^{-n} V^n.$$
This concludes since $h$ is $(-1)^n$ positive definite on $V^n$ and since $T$ is an anti-isometry.
\end{proof}

\subsubsection{Characterization of variations of classical Hodge structures}

Let $(\Kk,\Bb,\Tt,\Ww,D)$ be a variation of loop Hodge structures over a complex manifold $X$. If $\mu$ is a complex number of modulus one, one checks easily that $(\Kk,\Bb,\mu\Tt,\Ww,D)$ is also such a variation. This is called the \emph{circle action} on the variations of loop Hodge structures. In the correspondence given by theorem \ref{thm_fund}, this corresponds to an operation on harmonic bundles: one keeps the metric but changes the flat connection from $D=D_1$ to $D_\mu$, with the usual notations.

A variation of loop Hodge structures is a \emph{fixed point of this circle action} if it is isomorphic to all variations in its orbit.\\

Let $(\Vv,h,D,F^\bullet \Vv)$ be a variation of classical Hodge structures over a complex manifold $X$. By a family generalization of proposition \ref{CHS_NCHS}, this gives a family of loop Hodge structures $(\Kk,\Bb,\Tt,\Ww)$. We write $\tilde{D}$ for the naive connection induced by $D$ on $\Kk = L^2(S^1,\Vv)$.

\begin{prop}\label{class_nonab1}
 The quintuple $(\Kk,\Bb,\Tt,\Ww,\tilde{D})$ is a variation of loop Hodge structures over $X$, which is a fixed point of the circle action.
\end{prop}

\begin{proof}
 The $D$-flatness of $h$ implies the $\tilde{D}$-flatness of $\Bb$. Moreover, the right-shift operator $\Tt$ is $\tilde{D}$-flat as usual. It is thus enough to show that the differential constraints \ref{diff1}. and \ref{diff2} of definition \ref{VNCHS} are satisfied. Let $f$ be a local section of $\Ww$. It can be written in Fourier series as
$$ f = \sum_{n \in \Z} a_n \lambda^n,$$
where $a_n$ is a local section of $F^{-n}\Vv$. Then $\tilde{D}f$ is given by
$$ \tilde{D}f = \sum_n (Da_n) \lambda^n = T^{-1} \sum_n (Da_n) \lambda^{n+1}.$$
By the Griffiths transversality condition, $Da_n$ is a local section of $T^\ast X \otimes F^{-n-1}\Vv$, hence $\tilde{D}f$ is a local section of $T^\ast X \otimes \Tt^{-1}\Ww$. The differential constraint on the $(0,1)$-part of $\tilde{D}$ is proved in the same way.

Let $\mu$ be in $S^1$ and consider the map $\phi$ from $\Kk$ to $\Kk$ given by
$$
\phi(f)(\lambda) = f(\mu \lambda).
$$
Then, $\phi$ is a flat isometry of $\Kk$ satisfying $\mu \Tt \circ \phi = \phi \circ \Tt$. This shows that the variation is a fixed point of the circle action.
\end{proof}

This proposition has the following converse (compare with lemma 4.1 in \cite{Sim_higloc}):

\begin{thm}\label{class_nonab}
 Let $(\Kk,\Bb,\Tt,\Ww,D)$ be a variation of loop Hodge structures, which is a fixed point of the circle action. Then, it is canonically isomorphic to a variation of loop Hodge structures coming from a variation of classical Hodge structures, by the construction of proposition \ref{class_nonab1}.
 %If a non-classical variation of Hodge structures is a fixed point of the circle action, then it is canonically isomorphic to a non-classical variation coming from a classical variation, as described above. More precisely, it is sufficient that the non-classical variation $(H,T,D,W)$ is isomorphic to $(H,\mu T,D,W)$, for a single $\mu$ in $S^1$, that is not a root of unity.
\end{thm}

\begin{proof}
 See theorem 1.2.3 in \cite{my_thesis}.
\end{proof}

\section{The period domain} \label{sec_per_dom}

In this section, we construct a \emph{period domain} for variations of loop Hodge structures. It is very similar to the period domains encountered in classical Hodge theory, though it is modeled on a Hilbert space. This construction appeared before in different degrees of generalization, in particular in \cite{DPW} and \cite{EschTrib}. To a harmonic bundle, we can then associate a period map from (the universal cover of $X$) to the period domain; in this way, one can consider holomorphic maps, when studying harmonic bundles.

\subsection{Construction of the period domain} \label{constr_dom}

We fix some notations that will be used in the following. The letter $G$ stands for the general linear group $GL(n,\C)$, where $n$ is some positive integer. Its maximal compact subgroup $U(n)$ is denoted by $K$. We use the Gothic letters $\gg$ and $\kk$ for the Lie algebras of $G$ and $K$. 

The Cartan involution of $G$, relative to $K$, is written $\sigma$: it is given by $\sigma(A) = (A^\ast)^{-1}$, where $A^\ast$ is the transconjugate of $A$. We also write $\sigma$ for the Cartan involution on $\gg$; the corresponding Cartan decomposition is $\gg = \kk + \pp$, where $\pp$ is the subspace of Hermitian matrices in $\gg$.

\subsubsection{Loop algebra} \label{loop_alg}

On $\gg$, we define the usual Hermitian inner product:
$$ (A,B)_\gg := \Tr(AB^\ast).$$
The corresponding norm will be written $|.|_\gg$ or $|.|$ if there is no possible confusion.

Let $s$ be a nonnegative real number. The space of loops with regularity $H^s$ and with values in $\gg$ is by definition:
$$
\Lambda^s \gg = \{\sum_{n \in \Z} a_n \lambda^n \mid a_n \in \gg, \sum_{n \in \Z} (1+|n|^2)^s |a_n|_\gg^2 < \infty\}.
$$

In particular, these loops are in $L^2(S^1,\gg)$. We define a Hermitian inner product on $\Lambda^s \gg$ by
$$ (\sum a_n \lambda^n,\sum b_n \lambda^n)_{\Lambda^s \gg} := \sum (1+|n|^2)^s (a_n,b_n)_{\gg}.$$

This endows $\Lambda^s \gg$ with the structure of a complex Hilbert space. Moreover, a map $f\in L^2(S^1,\gg)$ is in $\Lambda^s \gg$ if and only if all its matrix coefficients are in $H^s(S^1,\C)$. We now assume that $s > \frac{1}{2}$ and recall the two fundamental facts (\cite{Tao}, proposition 1.1):
\begin{prop}\label{prop_sobo}
 If $s > \frac{1}{2}$, then $H^s(S^1,\C)$ injects continuously in $\mathcal{C}^{0}(S^1,\C)$, the space of continuous functions from the circle to $\C$. Moreover, the product of two functions in $H^s(S^1,\C)$ is still in $H^s(S^1,\C)$ and the product $H^s(S^1,\C) \otimes H^s(S^1,\C) \rightarrow H^s(S^1,\C)$ is continuous.
\end{prop}

The second fact shows that $\Lambda^s \gg$ is endowed with the structure of an associative algebra (coming from the associative algebra structure of $\gg$). In particular, it is equipped with a complex Hilbert-Lie algebra structure, that is a Lie algebra on a complex Hilbert space with continuous bracket for the Hilbert norm. Until remark \ref{pb_comp}, the real number $s > 1/2$ will be fixed and omitted in the notations of the different loop groups and loop algebras.

\subsubsection{Loop group}

Although this is quite unnatural, we see $G$ as an open set in its Lie algebra $\gg$. Since functions in $\Lambda \gg$ are continuous, the following definition makes sense:
$$\Lambda G := \{f \in \Lambda \gg \mid f(\lambda) \in G, \forall \lambda \in S^1\}.$$
This subset is an open set of $\Lambda \gg$ since the $H^s$-topology is finer than the compact-open topology (and since $S^1$ is compact). Hence, $\Lambda G$ is a complex Hilbert manifold. We claim that $\Lambda G$ is also a group. The stability by product works as for the Lie algebra; for the inverse, it is sufficient by Cramer's rule to remark that the determinant is bounded away from zero since it is a continuous non-vanishing function. Moreover the group operations are smooth for the Hilbert manifold structure. Indeed, since the product in $\Lambda G$ is the restriction of the product in $\Lambda \gg$, which is a bilinear map, it is sufficient to show that the product is continuous. This is the last part of proposition \ref{prop_sobo}.

The Lie algebra of $\Lambda G$ can be identified with $\Lambda \gg$ and one can show that the exponential map is the pointwise exponential map from $\gg$ to $G$. In summary:

\begin{prop} \label{loop_group}
 The space $\Lambda G$ is a complex Hilbert-Lie group with complex Hilbert-Lie algebra $\Lambda \gg$.
\end{prop}

\subsubsection{Twisted loop group}

We recall that $\sigma$ denotes the Cartan (anti-linear) involution of $G$ relative to $K = U(n)$. In the loop Lie group or in the loop Lie algebra, one defines the corresponding (anti-linear) involution:
$$
\hat{\sigma}: \gamma(\lambda) \mapsto \sigma(\gamma(-\lambda)).
$$
The closed subgroup (resp. subalgebra) of fixed points of $\hat{\sigma}$ is denoted by $\Lambda_\sigma G$ (resp. $\Lambda_\sigma \gg$). Using the Fourier series decomposition, one has
$$
\Lambda_\sigma \gg = \{\sum a_n \lambda^n \mid \sigma(a_n) = (-1)^n a_{-n} \}.
$$
This means that the loops in $\Lambda_\sigma \gg$ are exactly the sums of an even loop with values in $\kk$ and an odd loop with values in $\pp$. In particular, $\Lambda_\sigma \gg$ is a real form of $\Lambda \gg$ and the group $\Lambda_\sigma G$ is a real Hilbert-Lie group, which is a real form of $\Lambda G$.

\subsubsection{The period domain and its compact dual}

Let $\Lambda^+ \gg$ be the subalgebra of $\Lg$ of functions with only nonnegative Fourier coefficients and let $\Lambda^+ G$ be the subgroup of $\LG$ of functions with only nonnegative Fourier coefficients. It is a complex Hilbert-Lie subgroup of $\LG$, with Lie algebra $\Lambda^+ g$.
\begin{prop}\label{symm_princ}
 The intersection $\Lambda_\sigma G \bigcap \Lambda^+ G$ equals $K$, where an element $k$ of $K$ is viewed as the constant map from $S^1$ to $k$.
\end{prop}

\begin{proof}
Since a loop $\gamma$ in $\Lambda^+ G$ has no negative Fourier coefficients, it defines a holomorphic function $f$ on the open unit disk that extends continuously on the boundary. Let $z$ be in $\CP^1$ with $|z| > 1$ and define $f(z) = \sigma(f(-\frac{1}{\bar{z}}))$. Then $f$ is continuous on $\CP^1$ if $\gamma$ is in $\Lambda_\sigma G$; moreover $f$ is holomorphic in the outer disk $|z| > 1$. In fact, this extension $f$ is the canonical extension given by a generalization of Schwarz reflection principle (\cite{Ahlf}, section 6.5, theorem 24), which implies that $f$ is holomorphic on the whole sphere.

More precisely, consider the complex group $G \times G$ and the anti-linear involution $\theta$ defined by
\begin{align*}
 \theta: G \times G &\rightarrow G \times G, \\
 (A,B) &\mapsto (\sigma(B),\sigma(A)).
\end{align*}

Embed $\Lambda G$ in $\Lambda (G \times G)$ by the map $\iota: \gamma(\lambda) \mapsto (\gamma(\lambda),\gamma(-\lambda))$. Let $\gamma$ be in $\Lambda_\sigma G \cap \Lambda^+ G$. Then, $\iota(\gamma)$ has a holomorphic extension to the open unit disk, extends continuously to the closed unit disk and, on the unit circle, it takes values in the real form of $G \times G$ defined by $\theta$. By Schwarz reflection principle, such a map extends to a holomorphic map from $\CP^1$ to $G \times G$, which has to be constant. Hence, $\gamma$ is constant and its value is necessarily in $K$; the other inclusion is obvious.
\end{proof}

 Another self-contained proof will be given at then end of the proof of proposition \ref{thm_act}.

\begin{prop}\label{open_dom}
 The multiplication map $\Lambda_\sigma G \times \Lambda^+ G \rightarrow \Lambda G$ is open.
\end{prop}

\begin{proof}
 It is enough to show that the equality $\Lambda \gg = \Lambda_\sigma \gg + \Lambda^+ \gg$ holds at the Lie algebra level since, in Banach-Lie groups, the exponential map is a local diffeomorphism. Let $\gamma = \sum_{n \in \Z} a_n \lambda^n$ be a loop in $\Lambda \gg$. Define
 $$\gamma_\sigma := \sum_{n \in \Z} \tilde{a}_n\lambda^n,$$
 with $\tilde{a}_n = a_n$ for $n \leq 0$ and $\tilde{a}_n = (-1)^n \sigma(a_{-n})$ for $n \geq 1$ and
 $$\gamma_+ = \sum_{n \in \N} b_n \lambda^n,$$
 with $b_n = a_n - \tilde{a}_n$, for $n \geq 1$.

 Then, by construction $\gamma = \gamma_\sigma + \gamma_+$, $\gamma_\sigma$ is in $\Lambda_\sigma \gg$ and $\gamma_+$ is in $\Lambda^+ \gg$.
\end{proof}

\begin{comment}
\begin{rem}
 About the connected components of these groups: let $\Omega G$ be the based loop space of $G$ and let $LG$ be its free loop space. Both spaces are endowed with the compact open topology. It is well known that $\pi_i(\Omega G) = \pi_{i+1} G$ and that, as topological spaces, $LG \cong \Omega G \times G$. In particular, $\pi_0(LG) \cong \pi_1(G) \times \pi_0(G)$, hence $\pi_0(LG) \cong \Z$. By some approximation results, this can be extended to the space $\Lambda G$, of $H^s$-loops, so that $\pi_0(\Lambda G) = \Z$. 

Loops in $\Lambda^+ G$ extend continuously to the whole disk, so they are null-homotopic. On the other hand, I claim that $\Lambda_\sigma G$ encounters exactly the connected components containing the loops $\lambda \mapsto \lambda^k \Id$ for even $k$. One inclusion is clear since these loops are in $\Lambda_\sigma G$; for the other inclusion, one can use the fact that the connected component in which a loop $\gamma$ lives depends only on the homotopy type of the loop
 $$\tilde{\gamma}: \lambda \mapsto \frac{\det(\gamma(\lambda))}{|\det(\gamma(\lambda))|},$$
 with values in $S^1$. If $\gamma$ is in $\Lambda_\sigma G$, then the loop $\tilde{\gamma}$ is even, which proves the claim.
\end{rem}
\end{comment}

\begin{defn}\label{per_dom}
 The homogeneous space $\Dd := \Lambda_\sigma G/K$ is the \emph{loop period domain} for the symmetric space $G/K$. The homogeneous space $\Ddc := \Lambda G/\Lambda^+ G$ is its \emph{compact dual}. The terminology comes from the corresponding spaces in classical Hodge theory.%, as explained in the introduction and in section \ref{VHS}.
\end{defn}

\begin{rem}\label{rem_avt_disc}
 By propositions \ref{symm_princ} and \ref{open_dom}, the period domain $\Dd$ embeds as an open set in its compact dual $\Ddc$.
\end{rem}

\subsubsection{Complex structure on the period domain}

Since the period domain $\Dd$ is open in its compact dual $\check{\Dd}$ which is a complex Hilbert manifold, it inherits a complex structure. Let us see a more intrinsic way of viewing this complex structure.

The group $K$, seen as a group of constant loops, acts on the Lie algebra $\Lambda_\sigma \gg$ by the adjoint action. It preserves its Lie algebra $\kk$ and its orthogonal
$$\Lambda_\sigma^0 \gg := \{\gamma \in \Lambda_\sigma \gg \mid \gamma = \sum_{n \neq 0} a_n \lambda^n\}$$
for the (real) Hilbert structure on $\Lambda_\sigma \gg$ coming from the complex Hilbert structure on $\Lambda \gg$. Hence, as a $K$-module, the quotient Hilbert space $\Lambda_\sigma \gg/\kk$ is isomorphic to $\Lambda_\sigma^0 \gg$.

In the same way, $\Lambda \gg/\Lambda^+ \gg$ is isomorphic as a $K$-module to
$$\Lambda^- \gg := \{\gamma \in \Lambda \gg \mid \gamma = \sum_{n < 0} a_n \lambda^n\}.$$

The spaces $\Lambda_\sigma^0 \gg$ and $\Lambda^- \gg$ are the tangent spaces at the (same) base point of $\Dd$ and $\Ddc$. Hence, by the proof of proposition \ref{open_dom}, there is a canonical isomorphism between these two spaces. In particular, the complex structure on $\Lambda^- \gg$ gives a complex structure $J$ on $\Lambda_\sigma^0 \gg$. One computes that for $\gamma = \sum_{n \neq 0} a_n \lambda^n \in \Lambda_\sigma^0 \gg$,
\begin{eqnarray}\label{compl_str}
J\gamma = i\sum_{n < 0} a_n \lambda^n + (-i) \sum_{n > 0} a_n \lambda^n.
\end{eqnarray}

It is of course $K$-invariant. Since the tangent space of $\Dd$ is isomorphic as a $\Lambda_\sigma G$-bundle to
$$
T \Dd \cong \Lambda_\sigma G \times_K \Lambda_\sigma^0 \gg,$$
this provides an almost complex structure $J$ on $T \Dd$, which is necessarily integrable since it coincides with the complex structure coming from the open inclusion in the compact dual. The holomorphic tangent bundle is thus given by
$$
T_{\text{hol}} \Dd = \Lambda_\sigma G \times_K \Lambda^- \gg.
$$

\subsubsection{Horizontal structure on the period domain}

The complex subspace $\Lambda^{\geq -1} \gg = \{\gamma \in \Lambda \gg \mid \gamma = \sum_{n \geq -1} a_n \lambda^n\}$ is stable by the adjoint action of $\Lambda^+ G$. The \emph{holomorphic horizontal bundle} is defined to be
$$
T_{h,\text{hol}} \Ddc := \Lambda G \times_{\Lambda^+ G} \Lambda^{\geq -1} \gg/\Lambda^+ \gg \subset T_\text{hol} \Ddc.
$$
This is a finite-dimensional $\Lambda G$-equivariant holomorphic vector bundle over $\Ddc$. In particular, it gives a $\Lambda_\sigma G$-equivariant holomorphic vector bundle over $\Dd$.

A holomorphic map with target $\Dd$ or $\Ddc$ will be called \emph{horizontal} if its differential sends the holomorphic tangent bundle to the holomorphic horizontal tangent bundle.

\subsubsection{The Grassmannian of outgoing subspaces}

In the following, we use the letter $H$ for Krein spaces, in order to avoid confusion with the unitary group $K$.\\

Let $(H,B,T)$ be an outgoing Krein space of virtual dimension $n$ and denote by $\Gr(H)$ the \emph{outgoing Grassmannian} of $(H,B,T)$, that is the set of all outgoing subspaces. Let $W$ be a base-point in $\Gr(H,T)$. By proposition \ref{kr_out}, one can assume that $H = L^2(S^1,\C^n)$ with its natural outgoing Krein structure and $W$ is $L_+^2(S^1,\C^n)$.

Consider the set of essentially bounded mesurable maps from $S^1$ to $G$, where $G$ is embedded in $\gg$. This set is stable by the pointwise product but the pointwise inverse of an essentially bounded map need not be essentially bounded (compare with the discussion before proposition \ref{loop_group}). Hence, we define $\Lambda^\infty G$ to be the set of essentially bounded measurable maps whose pointwise inverse are also essentially bounded, modulo the equivalence relation of equality almost everywhere. The requirement on the inverse is equivalent to asking that $\det(g(\lambda))^{-1}$ is an essentially bounded function. As for the $H^s$-loops, we can give a structure of complex Banach-Lie group to $\Lambda^\infty G$, with Lie algebra $\Lambda^\infty \gg$. We can also define the subgroup $\Lambda_\sigma^\infty G$, which is a real Banach-lie subgroup with Lie algebra $\Lambda_\sigma^\infty \gg$.

The group $\Lambda^\infty G$ acts on $H$ by pointwise matrix multiplication and gives an embedding $\Lambda^\infty G \hookrightarrow \Aut(H)$ of complex Banach-Lie groups. 

\begin{prop}\label{thm_act}
The group $\Lambda_\sigma^\infty G$ acts transitively on $\Gr(H)$, and the stabilizer of $W$ equals the subgroup $K$ of constant loops in $K$.
\end{prop}

\begin{proof}
By propositions \ref{comm_shift} and \ref{B_isom}, the group $\Lambda^\infty_\sigma G$ acts on $H$ and can be identified with the group of elements that commutes with $T$ and preserves the Krein metric $B$. By the definition \ref{sub_out} of an outgoing subspace, $\Lambda_\sigma^\infty G$ preserves the outgoing Grassmannian. 

Let $\tilde{W}$ be an outgoing subspace. 
By lemma \ref{isom_Hilb}, the orthocomplement $V$ of $T\tilde{W}$ in $\tilde{W}$ has dimension $n$. Let $\phi: \C^n \rightarrow V$ be an isometry (with respect to the restriction 
of the Krein metric). There is a unique way to extend $\phi$ to a map from $H$ to $L^2(S^1, V) \cong H$ that commutes with the outgoing operator (see also the proof of Theorem 8.3.2
in \cite{PS}). This extension, still denoted $\phi$, commutes with $T$ by construction, and is a Krein isometry: indeed, it sends $T^i e_j$ to $T^i \phi(e_j)$, both having square norm
equal to $(-1)^i$, the elements $T^i e_j$ are $B$-orthogonal and they give a Hilbert basis of $H$. Hence, by propositions \ref{comm_shift} and \ref{B_isom}, $\phi$ lives in
$\Lambda_\sigma^\infty G$ and sends $\C^n$ on $V$. Hence, it sends $W$ on $\tilde{W}$.\\

The assertion about the stabilizer is essentially proposition \ref{symm_princ} but we give another self-contained proof. Let $\gamma$ be in $\Lambda_\sigma^\infty G$, stabilizing the Fourier-nonnegative subspace. Writing Fourier series representation, this implies that $\gamma$ has its Fourier series in nonnegative degrees. We claim that its Fourier series also has to be concentrated in nonpositive degrees. Indeed, since $\gamma^{-1}$ is also holomorphic in the open unit disk, we can write
$$\gamma^{-1} = \sum_{n \in \N} A_n \lambda^n.$$
Then $\hat{\sigma}(\gamma)$ is given by
$$\hat{\sigma}(\gamma) = \sum_{n \in \N} A_n^\ast (-\lambda^{-n}),$$
since the conjugate of $\lambda$ in $S^1$ is $\lambda^{-1}$. Since $\gamma = \hat{\sigma}(\gamma)$, this proves the claim: $\gamma$ has to be constant and necessarily in $K$.
\end{proof}

\begin{rem}\label{pb_comp}
 This proposition gives a bijection between $\Gr(H)$ and the homogenous space $\Lambda_\sigma^\infty G/K$. This enables us to give a Banach manifold structure to $\Gr(H)$. In order to distinguish the different regularities discussed, we write again the superscript $s$ for loops with regularity $H^s$.
\end{rem}

 The complex structure defined above for $\Lambda_\sigma^s G/K$ is not well-defined on $\Lambda_\sigma^\infty G/K$. Indeed, let $f$ be in $\Lambda_\sigma^{0,s} \gg$. By equation (\ref{compl_str}), $Jf$ is given by $if_- - if_+$, where $f_-$ and $f_+$ are the Fourier negative and positive parts of $f$. If we only assume that $f$ is in $\Lambda_\sigma^\infty \gg$, then $Jf$ is in $\Lambda_\sigma^\infty \gg$ if and only if $f_-$ and $f_+$ are in $\Lambda^\infty \gg$. We claim that this is not true in general.

 Indeed, let $M$ any matrix in $\pp$ and consider the function $f: S^1 \rightarrow \gg$, such that $f(\lambda) = M$ if $\Im \,\lambda > 0$ and $f(\lambda) = - M$ if $\Im \,\lambda \leq 0$. This is a \emph{square signal function} which belongs to $\Lambda_\sigma^\infty \gg$. Its Fourier series are given by
 $$ f = \sum_{n \text{ odd}} \frac{1}{n} \lambda^n,$$
 up to some multiplicative factor. Then, $f_-$ and $f_+$ are not essentially bounded. The issue here seems to  be that $\Lambda^\infty_+ G$ is not a closed Lie subgroup of $\Lambda^\infty G$, in the sense of  \cite{Lan}, section VI.5. Indeed, $\Lambda^\infty_+ \gg$ does not admit a complementary closed subspace in $\Lambda^\infty \gg$.

 However one can show that all this construction works in the slightly more general regularity $H^{1/2} \cap L^\infty$ (\cite{PS}, Chapter 8). Thanks to lemma \ref{reg_permap}, it will be sufficient for our purpose to consider the $H^s$-regularity.

\subsection{The period map} \label{per_map}

The period map of a harmonic bundle is the developing map of its associated variation of loop Hodge structures. In this subsection, we define this map and show that it shares the properties of the period maps in classical Hodge theory.

\subsubsection{Developing map of a variation of loop Hodge structures}

Let $(\Hha,\Bb,\Tt,\Ww,D) \rightarrow X$ be a variation of loop Hodge structures of virtual dimension $n$ and denote by $\pi:\tilde{X} \rightarrow X$ the universal cover of $X$. Let $\tilde{x}_0$ be a base-point in $\tilde{X}$ and consider the pullback of the variation of loop Hodge structure to $\tilde{X}$. Its fiber over $\tilde{x}_0$ can be identified with $H:= L^2(S^1,\C^n)$ with its canonical outgoing Krein structure such that $(\pi^\ast W)_{\tilde{x}_0}$ corresponds to the Fourier-nonnegative subspace.

Since $\pi^\ast D$ is a flat connection over a simply connected manifold, one can trivialize the variation over $\tilde{X}$. In this way, the variation on $\tilde{X}$ is carried by the trivial bundle $\tilde{X} \times H$, the Krein metric and outgoing operator are constant and $W$ is a subbundle of outgoing subspaces which is the Fourier-nonnegative subspace over $\tilde{x}_0$. The datum of this subbundle is equivalent to a map $f: \tilde{X} \rightarrow \Lambda_\sigma^\infty G/K$, such that $f(\tilde{x}_0) = eK$ by proposition \ref{thm_act}.

We write $\rho_{\tot}: \pi_1(X) \rightarrow \Aut(H)$ for the monodromy of the flat connection $D$. Since $\Bb$ and $\Tt$ are $D$-flat, $\rho_{\tot}$ takes its values in $\Lambda_\sigma^\infty G$. The map $f$ is equivariant with respect to the monodromy representation $\rho_{\tot}$, since the data come from $X$. Moreover, $f$ is smooth by the very definition of a subbundle of outgoing subspaces in an outgoing Krein bundle. 

\begin{lem}\label{reg_permap}
 Let $s$ be any real number greater than $1/2$. Then $f$ takes its values in $\Lambda_\sigma^s G/K$ and the monodromy $\rho_{\tot}$ takes its values in $\Lambda_\sigma^s G$.
\end{lem}

\begin{proof}
 By theorem \ref{thm_fund}, the variation comes from a harmonic bundle and the flat connection $D$ comes from a circle of flat connections. This circle of flat connections depends smoothly on the parameter $\lambda$ in $S^1$ (even algebraically). Hence, by the theorem about smooth dependence of solutions of ODE, the parallel transport of $D$ stabilizes the smooth loops with values in $\C^n$. This proves that a local lift of $f$ to $\Lambda_\sigma^\infty G$ takes its values in the smooth loop group. By the Sobolev embedding theorem, a map is smooth if and only if it belongs to all $H^s$; this concludes the proof for $f$. The proof for the monodromy is analogous.
\end{proof}

Thanks to this lemma, we will not need to consider the space $\Gr(H) = \Lambda_\sigma^\infty G/K$ anymore; we choose a number $s > 1/2$ and we write again $\Lambda_\sigma G/K$ for the period domain of loops with regularity $H^s$.

\begin{rem}\label{rem_reg}
 The same proof shows that the map $f$ takes its values in $\Lambda_\sigma^h G/K$, where $\Lambda_\sigma^h G$ is the subgroup of $\Lambda_\sigma G$ of loops which are restrictions of holomorphic maps from $\C^\ast$ to $G$.
\end{rem}

\begin{defn}
 The developing map $f: \tilde{X} \rightarrow \Dd$ of a variation of loop Hodge structures is called the \emph{period map} of the variation.
\end{defn}

\subsubsection{The tangent bundle of the period domain}

The following notations will be used in the whole section. The space $L^2(S^1,\C^n)$ with its structure of loop Hodge structure is written $H$. The period domain is $\Dd = \Lambda_\sigma G/K$ and the Lie algebra of $\LsG$ is $\Lsg$. When there is no possible confusion, the trivial bundles $Y \times H$ and $Y \times \Lsg$ over a complex manifold $Y$ will be simply written $H$ and $\Lsg$.

\begin{defn}
 The \emph{tautological bundle} over $\Dd$ is the complex Hermitian bundle of rank $n$ defined by $(H^{0,\Dd})_{W \in \Dd} = W \ominus^\perp TW$, where the orthogonal is taken with respect to the Krein metric. The \emph{tautological decomposition} of $H$ over $\Dd$ is the Hilbert sum decomposition
$$ H = \widehat{\bigoplus}_{i \in \Z} H^{i,\Dd}, $$
where $H^{i,\Dd} := T^i H^{0,\Dd}$. The \emph{tautological decomposition} of $\Lsg$ over $\Dd$ is the Hilbert sum decomposition
$$
 \Lsg = \widehat{\bigoplus}_{i \in \Z} \Lsg[i,\Dd],$$
where $\Lsg[i,\Dd] := \{f \in \Lsg \mid f(H^{0,\Dd}) \subset H^{i,\Dd}\} = \{f \in \Lsg \mid \forall j, f(H^{j,\Dd}) \subset H^{i+j,\Dd}\}$.
\end{defn}

\begin{prop}\label{tang_dom}
 The tangent bundle of $\Dd$ is canonically isomorphic to the bundle $\Lsg / \Lsg[0,\Dd]$.
\end{prop}

\begin{proof}
 We recall that the tangent bundle of $\Dd$ is isomorphic to $\LsG \times_K \Lsg/\Lsg[0]$, as a $\LsG$-equivariant bundle. Here, $\Lsg[0]$ is the subspace of $\Lsg$ of functions with vanishing zero Fourier coefficient. In this representation, a point in $T\Dd$ over $g.o$, where $g$ is in $\LsG$ and $o = eK$ is the base-point, is given by
 $$ [g : f + \Lsg[0]],$$
 where $f$ is in $\Lsg$ and the equivalence relation is given by $[g : f + \Lsg[0]] = [gk : \Ad(k)^{-1} f + \Lsg[0]]$.\\

 We define a vector bundle map from $T \Dd$ to $\Lsg / \Lsg[0,\Dd]$ by
 $$ [g : f + \Lsg[0]] \mapsto \Ad(g)f + \Lsg[0,\Dd],$$
 over $g.o$. This is well-defined. Indeed, on the one hand, if one changes $g$ by $gk$ and $f$ by $\Ad(k^{-1}f)$, then $gk.o = g.o$ and $\Ad(gk)\Ad(k^{-1})f = \Ad(g)f$. On the other hand, one has the equality $(\Lsg[0,\Dd])_{g.o} = \Ad(g)\Lsg[0]$. Since the inverse of this vector bundle map is given by $\Ad(g^{-1})$ over $g.o$, this is an isomorphism.\\
\end{proof}

\begin{lem}
 In the identification $\Lsg / \Lsg[0,\Dd] \cong \widehat{\bigoplus}_{j \neq 0} \Lsg[j,\Dd]$, the complex structure on $T\Dd$ is given by multiplication by $i$ on $\widehat{\bigoplus}_{j < 0} \Lsg[j]_\Dd$ and by $-i$ on $\widehat{\bigoplus}_{j > 0} \Lsg[j]_\Dd$. The horizontal subbundle is given by $\Lsg[-1] \oplus \Lsg[1]$.
\end{lem}

\begin{proof}
 This is true at the base-point and both the complex structure and the horizontal subbundle are stable by the $\LsG$ action. By the explicit isomorphism given in the proof of proposition \ref{tang_dom}, this concludes the proof.
\end{proof}

\subsubsection{Differential of a smooth map to the period domain} \label{dif_permap}

Let $f: X \rightarrow \Dd$ be a smooth map. The trivial bundle $H$ over $X$ has a canonical decomposition $H = \widehat{\bigoplus}_{i \in \Z} H^{i,X}$ given by pullback of the decomposition of $\Dd \times H$. We also have a decomposition of the trivial bundle $X \times \Lsg$. Let $d$ be the trivial connection on the bundle $H$ over $X$. We denote by $\pi$ the vector bundle projection $\pi: H \rightarrow H/H^{0,X}$. Let 
$$\iota: \Lsg/\Lsg[0,X] \cong \Hom(H^{0,X},H/H^{0,X})$$
be the canonical isomorphism of vector bundles, given by
$$\iota: f+ \in \Lsg \mapsto (h \in H^{0,X} \mapsto f(h) + H^{0,X}).$$
Finally, let $\beta$ be the differential operator $\pi \circ d_{|H^{0,X}}$ from $H^0_X$ to $H/H^0_X$.

\begin{prop}\label{beta}
The differential operator $\beta$ is a tensor, that is a $1$-form with values in $\Hom(H^{0,X}, H/H^{0,X})$. Moreover, if one considers the differential of $f$ to be a $1$-form with values in $\Lsg/\Lsg[0,X]$ (by proposition \ref{tang_dom}), then
$$ \beta = \iota \circ df.$$
\end{prop}

\begin{proof}
Let $g$ be a local lift of $f$ to $\LsG$. We write $\omega$ for the left Maurer-Cartan form in $\LsG$. In the homogeneous representation $T\Dd = \LsG \times_K \Lsg/\Lsg[0]$, $df$ is given by
$$ d_x f = [g(x) : (g^\ast \omega)_x + \Lsg[0]].$$

By the proof of proposition \ref{tang_dom}, as a $1$-form with values in $\Lsg/\Lsg[0,\Dd]$, $df$ is thus given by
$$ d_x f = \Ad(g(x)) (g^\ast \omega)_x + (\Lsg[0,\Dd])_x.$$

On the other hand, we need to compute $\beta$. Let $s$ be a local smooth section of $H^{0,X}$. It can be written $s(x) = g(x)h(x)$, where $h$ is a smooth map with values in $\C^n \subset H$. Then
$$ d_x s = (d_x g)h(x) + g(x)(d_x h).$$
Hence modulo $(H^{0,X})_x$, $d_x s = (d_x g)h(x) = (d_x g)g^{-1}(x)s(x)$. In particular, $\beta$ is a tensor, given by
$$ \beta_x(s(x)) = (d_x g)g^{-1}(x)s(x) + (H^{0,X})_x,$$
as a $1$-form with values in $\Hom(H^{0,X},H/H^{0,X})$.

Since $\iota \circ df$ is given by 
$$ (\iota \circ df)_x (s(x)) = \Ad(g(x)) (g^\ast \omega)_x (s(x)) + (H^{0,X})_x,$$
it is sufficient to show that $\Ad(g(x))(g^\ast \omega)_x = (d_x g)g^{-1}(x)$. This follows from $(g^\ast \omega)_x = g(x)^{-1} d_x g$ and concludes the proof.
\end{proof}

\subsubsection{Differential of the period map}

Let $(\Ee,D,h)$ be a harmonic bundle over a simply-connected manifold $X$ and let $f: X \rightarrow \Dd$ be the period map. Let $(L^2(S^1,\Ee),L^2_+(S^1,\Ee),\tilde{D})$ be the associated variation of loop Hodge structures. By parallel transport, there is an isomorphism of variations of loop Hodge structures
$$
\gamma : (L^2(S^1,\Ee),\tilde{D}) \rightarrow (H,d)
$$
where the outgoing subbundle in $H$ is given by $\widehat{\bigoplus}_{i \in \N} H^i_{\tilde{X}}$. The differential of $f$ takes its values in the bundle $\Hom(H^{0,X},H/H^{0,X})$ which, by $\iota$, is isomorphic to $\Lsg/\Lsg[0,X]$. We omit the isomorphism $\iota$ in what follows.

On the other hand, the $1$-form $\alpha = \lambda^{-1}\theta + \lambda\theta^\ast$ takes its values in the bundle $\Lambda_\sigma\End(\Ee)/\Lambda_\sigma^0\End(\Ee)$. Here, $\Lambda_\sigma\End(\Ee)$ is the Hilbert subbundle of $H^s(S^1,\End(\Ee))$ whose elements are the loops $\gamma$ with values in $\End(\Ee)$, satisfying $\sigma(\gamma(\lambda)) = \gamma(-\lambda)$, where $\sigma$ is the Cartan involution induced from the harmonic metric $h$. The quotient $\Lambda_\sigma\End(\Ee)/\Lambda_\sigma^0\End(\Ee)$ is a complex Hilbert bundle.

The isomorphism $\gamma$ induces an isomorphism $\gamma: \Lambda_\sigma\End(\Ee)/\Lambda_\sigma^0\End(\Ee) \cong \Lsg/\Lsg[0,X]$ of complex Hilbert bundles.
\begin{prop}\label{expr_diff_per}
 The equality $df = \gamma \circ \alpha$ holds.
\end{prop}

\begin{proof}
 We recall that $\beta$ was defined by $\beta := \pi \circ d_{|H^0_X}$ and that $\beta = \iota \circ df$. Since the flat connection $d$ on $H$ is sent by $\gamma^{-1}$ to the flat connection $\tilde{D} = \nabla + \alpha$, the form $\gamma^{-1} \circ \beta$ is equal to the projection of $\tilde{D}$ on $\Lambda_\sigma\End(\Ee)/\Lambda_\sigma^0\End(\Ee)$, that is to $\alpha$. Hence, $\beta = \gamma \circ \alpha$, which concludes the proof, since we omit the isomorphim $\iota$.
\end{proof}

\begin{defn}
 Let $(\Ee,D,h)$ be a harmonic bundle over a complex manifold $X$. The \emph{total monodromy} of the harmonic bundle is the monodromy of the flat connection $\tilde{D}$ in its associated variation of Hodge structures. We write it $\rho_{\text{tot}}: \pi_(X) \rightarrow \Lambda_\sigma G$ (that it takes its values in this group comes from propositions \ref{comm_shift} and \ref{B_isom}).
\end{defn}

We can now show that the period map share the properties of the period maps in classical Hodge theory.

\begin{thm}\label{holo_permap}
 The period map induced from a harmonic bundle over $X$ is holomorphic, horizontal and $\rho_{\text{tot}}$-equivariant. Conversely, let $\rho_{\text{tot}}: \pi_1(X) \rightarrow \LsG$ and let $f: \tilde{X} \rightarrow \LsG/K$ be a holomorphic, horizontal and $\rho_{\text{tot}}$-equivariant map that satisfies $f(\tilde{x}_0) = eK$. Then, the datum $(H,d,\widehat{\bigoplus}_{i\in\N} H^i_{\tilde{X}})$ defines a variation of loop Hodge structures on $\tilde{X}$ that descends to one on $X$.
\end{thm}

\begin{proof}
The claim about the monodromy is clear so we can work on the universal cover $\tilde{X}$. Since the isomorphism $\gamma$ respects all the structures, it is sufficient to check that the $1$-form $\alpha$ lives in the horizontal subbundle of $\Lambda_\sigma\End(\Ee)/\Lambda_\sigma^0\End(\Ee)$ and commutes with the complex structure. As before, the horizontal subbundle is given by the loops in $\Lambda_\sigma\End(\Ee)$ that have only Fourier coefficients in degrees $-1$ and $1$. Hence, $\alpha$ is horizontal. Since $\theta$ (resp. $\theta^\ast$) is a holomorphic (resp. anti-holomorphic) section of $\End(\Ee)$, the computation of $J\alpha$ gives:
\begin{align*}
 J\alpha &= \lambda^{-1} J_{\End(\Ee)} \circ \theta - \lambda J_{\End(\Ee)} \circ \theta^\ast \text{\,\,\, by equation \eqref{compl_str}}\\
&= \lambda^{-1} \theta \circ J_{T\tilde{X}} + \lambda \theta^\ast \circ J_{T\tilde{X}} \\
&= \alpha \circ J_{T\tilde{X}},
\end{align*}
which proves the claim.

For the converse, the same analysis shows that $(H,d,\widehat{\bigoplus}_{i\in\N} H^i_{\tilde{X}})$ defines a variation of loop Hodge structures on $\tilde{X}$ and it descends to $X$ since the map $f$ is equivariant under $\rho_{\text{tot}}$.
\end{proof}

\subsubsection{Holomorphic differential}

We consider the trivial complex Hilbert bundle $\Lg$ over $\Dd$. There is a tautological filtration on $\Lg$ given by $(F^{i,\Dd} \Lg)_{W \in \Dd} = \{f \in \Lg \mid f(W) \subset T^i W\}$. This is a holomorphic filtration since it is homogenous with respect to $\LsG$: indeed, $F^{i,\Dd} \Lg = \LsG \times_K F^i \Lg$, where $F^i \Lg$ is the subspace of $\Lg$ of loops having Fourier coefficients in degrees $\geq i$.

The holomorphic tangent space of $\Dd$ is canonically isomorphic to $\Lg/F^{0,\Dd} \Lg$ and the horizontal subspace to $F^{-1,\Dd} \Lg/F^{0,\Dd} \Lg$. If $f: \tilde{X} \rightarrow \Dd$ is the period map of a harmonic bundle, the holomorphic differential of $f$ thus takes its values in $F^{-1,\tilde{X}} \Lg / F^{0,\tilde{X}} \Lg$, with obvious notations. As before, the isomorphism $\gamma$ induces an isomorphism of holomorphic vector bundles 
$$F^{-1,\tilde{X}} \Lg / F^{0,\tilde{X}} \Lg \cong L^{2,\geq -1}(S^1,\End(\Ee))/L^{2,\geq 0}(S^1,\End(\Ee)),$$ 
where $L^{2,\geq i}(S^1,\End(\Ee))$ is the subbundle of $L^2(S^1,\End(\Ee))$ of loops whose Fourier coefficients are $\geq i$. There is an isomorphism 
$$j: \End(\Ee) \rightarrow L^{2,\geq -1}(S^1,\End(\Ee))/L^{2,\geq 0}(S^1,\End(\Ee))$$ given by
$$ j: \nu \in \End(\Ee)) \mapsto \lambda^{-1}\nu + L^{2,\geq 0}(S^1,\End(\Ee)).$$

\begin{lem}
 The map $j$ is an isomorphism of holomorphic vector bundles, where $\End(\Ee)$ is endowed with its Higgs bundle structure. 
\end{lem}

\begin{proof}
 We consider an arbitrary harmonic bundle $(\Ff,D = \nabla + \theta + \theta^{\ast},h)$. The $\bar{\d}$-operator of $\Ff$, seen as a Higgs bundle, is $\nabla^{(0,1)}$. On the other hand, $L^2(S^1,\Ff)$ is endowed with a flat connection $\tilde{D} = \nabla + \lambda^{-1}\theta + \lambda\theta^\ast$. The induced $\bar{\d}$-operator is its $(0,1)$-part, that is $\nabla^{(0,1)} + \lambda \theta^\ast$. When acting on the quotient $L^{2,\geq -1}(S^1,\Ff)/L^{2,\geq 0}(S^1,\Ff)$, only the term $\nabla^{(0,1)}$ remains.

We apply this to the harmonic bundle $\Ff = \End(\Ee)$.
\end{proof}

Our last proposition easily follows from proposition \ref{expr_diff_per}.

\begin{prop}\label{id_diff}
 Under the isomorphism (of holomorphic vector bundles) $\End(\Ee) \cong F^{-1,\tilde{X}} \Lg / F^{0,\tilde{X}} \Lg$, the holomorphic differential of $f$ is the Higgs field $\theta$.
\end{prop}

\subsection{The Higgs foliation} \label{hig_fol}

If $(\Ee,\theta)$ is a Higgs bundle over a complex manifold $X$, then the kernel of $\theta$ defines a (possibly singular) holomorphic distribution in $X$. When $(\Ee,\theta)$ comes from a harmonic bundle, it is shown in \cite{Mok}, theorem 3, that this distribution is involutive. Here, we show that this distribution is in fact integrable.

\subsubsection{Holomorphic distributions}

First, we recall some facts about holomorphic distributions. A \emph{holomorphic distribution} $\Dd$ in $X$ is a subset of the holomorphic tangent bundle $TX$ which is locally the vanishing set of a finite number of holomorphic $1$-forms on $X$. Let $x$ be some point in $X$, and let $\Dd_x$ be the germ in $x$ of a holomorphic distribution. For any family $\omega = (\omega_1,\dots,\omega_p)$ of germs of holomorphic $1$-forms defining $\Dd_x$, we define $q_x(\omega)$ to be the codimension in $X$ of the vanishing set of the $p$-form $\omega_1 \wedge \dots \wedge \omega_p$, with the convention $q_x(\omega) = \infty$ if this set is empty. We write $q(\Dd_x)$ for the supremum of such $q_x(\omega)$ where $\omega$ is any family of germs of holomorphic $1$-forms defining $\Dd_x$. 

A distribution $\Dd$ is \emph{involutive} if it can locally be defined by holomorphic $1$-forms $\omega_1,\dots,\omega_p$ which satisfy $d\omega_i \wedge \omega_1 \wedge \dots \wedge \omega_p = 0$, for all $i$ in $\{1,\dots,p\}$. Dually, this is equivalent to the usual definition: locally, for every holomorphic vector fields $X$ and $Y$, tangent to $\Dd$, the bracket $[X,Y]$ is also tangent to $\Dd$. Equivalently, an involutive distribution is called a \emph{foliation}.

A distribution $\Dd$ is \emph{integrable} if it can locally be defined by exact holomorphic $1$-forms $df_1,\dots,df_p$, where the $f_i$ are local holomorphic functions. An integrable distribution is involutive and, in the regular case, that is when $q(\Dd_x) = \infty$ for any $x$, then the Frobenius theorem asserts that the converse is true. In the singular case, things are subtler but this still works if the singular locus is small enough:

\begin{thm}[Malgrange, \cite{Mal}]
 Let $\Dd$ be an involutive distribution in $X$ and assume that $q(\Dd_x) \geq 3$ for any $x$ in $X$. Then, $\Dd$ is integrable.
\end{thm}

\subsubsection{Higgs foliation}

Let $(\Ee,D,h)$ be a harmonic bundle over $X$ and let $\theta$ denote the induced Higgs field. The kernel of $\theta$ defines a holomorphic distribution in $X$ since $\theta$ is a holomorphic $1$-form with values in $\End(\Ee)$, for some holomorphic structure on $\Ee$. We now have the following:

\begin{thm}\label{integ_higgs}
 The holomorphic distribution defined by the Higgs field is integrable.
\end{thm}

I thank Philippe Eyssidieux, who suggested this application to me.

\begin{proof}
Locally in $X$, one can define a holomorphic period map $f: X \rightarrow \Dd = \LsG/K$. By theorem \ref{id_diff}, a holomorphic vector field is in the kernel of the Higgs field if and only if it is in the kernel of $df$. Since we work locally, we can assume that $f$ takes its values in a Hilbert space $H$. We write $(f_i)_{i \in \N}$ for the coordinates of $f$ relative to some Hilbert basis of $H$. The distribution is thus defined by the vanishing of an infinite number of holomorphic $1$-forms. We claim that, locally, we can consider only a finite number of holomorphic $1$-forms. Indeed, the germ of the distribution at $x$ is completely determined by the germs at $(x,0)$ of the functions $(df_i)_{i \in \N}$ that go from the total space $TX$ to $\C$. Since, the local ring of the sheaf of holomorphic functions in any complex manifold is Noetherian, we can consider only a finite number of $df_i$, proving the claim.
\end{proof}

\begin{rem}
 If $(\Ee,\theta)$ is a Higgs bundle over a complex manifold $X$, there is \emph{a priori} no reason that the distribution that it defines is involutive. Indeed, there is no differential assumptions about the Higgs field in the definition of a Higgs bundle. Nevertheless, if $X$ is a compact K\"ahler manifold, if $(\Ee,\theta)$ is polystable with vanishing Chern classes, then the theorem of Hitchin and Simpson asserts that it comes from a harmonic bundle, so that one can apply theorem \ref{integ_higgs}.
\end{rem}

\section{The determinant line bundle} \label{det_line}

In this section, we define a line bundle on the period domain and use it to get informations on a harmonic bundle, \emph{via} the period map. This section has greatly benefited from conversations with Bruno Klingler.

\subsection{Central extensions of Lie groups} \label{central_ext}

We want to define a central extension by the circle $S^1$ of the group $\LsG$. The analogous problem for $\Lambda K$ is discussed in detail in \cite{PS}, chapters $4$ et $6$, from which we collect some results.

Let $H$ be a (Banach-)Lie group and let $\tilde{H}$ be a central extension of $H$ by $S^1$. As a vector space, the Lie algebra $\tilde{\hh}$ can be written $\tilde{\hh} = \hh \oplus i\R$ and the Lie bracket is given by 
$$
[(\xi,\lambda),(\eta,\mu)] = ([\xi,\eta],\omega(\xi,\eta)),
$$
where $\omega$ is a skew-symmetric map from $\hh \times \hh$ to $i\R$ satisfying the \emph{cocycle condition}
$$
\omega([\xi,\eta],\zeta) + \omega([\zeta,\eta],\xi) + \omega([\eta,\xi],\zeta) = 0.
$$

Conversely, let $\omega$ be such a map. Then, $\omega$ defines a $H$-invariant $2$-form on $H$ (that we still denote $\omega$), and the cocycle condition is equivalent to the closedness of this form. 

\begin{prop}{(Corollary of proposition 4.4.2 in \cite{PS})}\label{cent_ext_coc}
 If $H$ is simply connected, then a cocycle $\omega$ is obtained from a central extension of $H$ by $S^1$ if and only if the cohomology class of the $2$-form $\frac i{2\pi} \omega$ is integral.
\end{prop}

In order to define a central extension by $S^1$ of $\LsG$, a first possibility is thus to define a cocycle on the Lie algebra $\Lsg$. We write $<X,Y> = -\Tr(XY)$ in $\gg$.

\begin{defn}\label{cocycle}
 The \emph{fundamental cocycle} on $\Lg$ is the skew-symmetric map with values in $\C$ defined by
 \begin{align*}
\omega(\xi,\eta) :&= \frac i{2\pi}\int_0^{2\pi} <\xi(\theta),\eta'(\theta)>d\theta .
\end{align*}
\end{defn}

\begin{rem}
 The skew-symmetry follows from an integration by parts. The cocycle condition is proved in section 4.2 of \cite{PS}, for the Lie algebra $\Lk$ and we conclude by complex linearity.
\end{rem}

\begin{prop}\label{rest_fund}
The restriction of $\omega$ to $\Lambda \kk$ and $\Lsg$ takes its values in $i\R$.
\end{prop}

\begin{proof}
If $\xi$ is in $\Lambda \kk$, then $\xi(\theta)^\ast = -\xi(\theta)$ and if $\eta$ is in $\Lambda_\sigma \gg$, then $\eta(\theta)^\ast = - \eta(-\theta)$. For $i=1,2$, we take $\xi_i$ in $\Lambda \kk$. Then
\begin{align*}
\overline{-2i\pi \omega(\xi_1,\xi_2)} &= \int_0^{2\pi}  <\overline{\xi_1(\theta)},\overline{\xi_2'(\theta)}>d\theta \\
&= \int_0^{2\pi} <\,^t \xi_1(\theta),\,^t \xi_2'(\theta)>d\theta \\
&= \int_0^{2\pi} <\,\xi_1(\theta),\,\xi_2'(\theta)>d\theta \\
&= -2i\pi \omega(\xi_1,\xi_2).
\end{align*}
The third line is obtained from the properties of the trace operator.

The computation is similar for $\Lsg$, since the measure $d\theta$ is invariant under $\theta \mapsto -\theta$.
\end{proof}

The group $SU(n)$ is simply-connected and this implies that $\Lambda SU(n)$ is simply-connected. By proposition \ref{cent_ext_coc}, in order to define the central extension $\widetilde{\Lambda SU(n)}$, it is sufficient to show that the $2$-form induced by the cocycle $\omega$ (restricted to $\Lambda \mathfrak{su}(n)$) has its cohomology class in $H^2(\Lambda SU(n),2i\pi\Z)$. This is done with propositions 4.4.4 and 4.4.5 in \cite{PS}. Then, one can define the central extension $\widetilde{\Lambda U(n)} = \widetilde{\Lambda K}$ of $\Lambda K$; this is section 4.7 of  \cite{PS}.

For the group $\LsG$, the construction of an analogue of the transgression map used in propositions 4.4.4 and 4.4.5 is not clear and we will define the central extension of $\LsG$ in another way.\\

A central extension by $\C^\ast$ of $\LG$ is defined in section 6.7 of \cite{PS}. This is done by pullback of a central extension $\C^\ast \rightarrow \widetilde{GL_{\text{res}}} \rightarrow GL_{\text{res}}$, where $GL_{\text{res}}$ is some subgroup of the group of linear automorphisms of the Hilbert space $L^2(S^1,\C^n)$, so that we have an embedding $\LG \hookrightarrow GL_{\text{res}}$. The construction of the central extension of $GL_{\text{res}}$ is done in section 6.6. After proposition 6.6.5, it is remarked that the same construction can be applied in order to obtain a $S^1$-extension of the group $U_{\text{res}}$ which is the subgroup of $GL_{\text{res}}$ of isometries of the Hilbert space $L^2(S^1,\C^n)$. Since, the group $\LK$ embeds in $U_{\text{res}}$, this gives a central extension by $S^1$ of $\LK$. In proposition 6.7.1, it is proved that this extension has the cocycle $\omega$ defined above; hence this gives a concrete description of $\widetilde{\Lambda K}$.

Considering the space $L^2(S^1,\C^n)$ as a Krein space, we can do the same construction with the Krein isometry group instead of the Hilbert isometry group. The only point to check, which is implicit in section 6.6, is that an operator of determinant class which is unitary (for either the Hilbert or Krein metric) has determinant of modulus one. This is proved as in the finite-dimensional setting, using that if $M$ has a determinant and if $X$ is invertible, then $XMX^{-1}$ has a determinant too, equal to $\det(M)$.\\

A third possibility to define the central extension of $\LsG$ is less natural but easier, since it does not involve the description of the central extension by $\C^\ast$ of $\LG$. Let $1 \rightarrow \C^\ast \rightarrow \Ee \rightarrow \LsG \rightarrow 1$ be the pullback of the central extension by $\C^\ast$ of $\LG$. The group $\C^\ast$ decomposes as $\C^\ast = S^1 \times \R_+^\ast$; we can quotient the extension $\Ee$ by $\R_+^\ast$, giving a central extension
$$
1 \rightarrow S^1 \rightarrow \widetilde{\LsG} \rightarrow \LsG \rightarrow 1.
$$
The cocycle of the extension $1 \rightarrow \C^\ast \rightarrow \widetilde{LG} \rightarrow LG \rightarrow 1$ is the fundamental cocycle of definition \ref{cocycle}. This is essentially proposition 6.7.1 of \cite{PS}. Hence, the cocycle of $\widetilde{\LsG}$ is the restriction to the fundamental cocycle to $\Lsg$, followed by projection onto $i\R $ (with respect to the decomposition $\C = i\R \oplus \R$). From proposition \ref{rest_fund}, the fundamental cocycle restricted to $\Lsg$ already takes its values in $i\R$.\\

We sum up this discussion with the following proposition:
\begin{prop}
 There is a central extension
 $$
1 \rightarrow S^1 \rightarrow \widetilde{\LsG} \rightarrow \LsG \rightarrow 1,
$$
whose cocycle is given by definition \ref{cocycle}.
\end{prop}

\begin{rem}
 The restriction of $\widetilde{\LsG}$ over the subgroup $K$ of $\LsG$ is trivial; indeed, the fundamental cocycle vanishes on $\kk \otimes \kk$. In the following, we write $\tilde{K}$ for $K \times S^1$.
\end{rem}

\subsection{The determinant line bundle and its curvature} \label{curvature}

Another homogeneous representation of the loop period domain is given by $\Dd = \widetilde{\LsG}/\tilde{K}$. At the level of Lie algebras, we can write $\widetilde{\Lambda_\sigma \gg} = \tilde{\kk} \oplus \Lambda^0_\sigma \gg$, where $\Lambda^0_\sigma \gg$ is the subspace of $\Lambda_\sigma \gg$ of functions whose zeroth Fourier coefficient vanishes. A $\widetilde{\Lambda_\sigma G}$-equivariant complex line bundle on $\Dd$ is obtained from a character of $S^1 \times K$. 

\begin{defn}
The \emph{determinant line bundle} over $\Dd$ is the complex line bundle given by the character $(k,\lambda) \mapsto \lambda$ of $\tilde{K}$.
\end{defn}

As a homogeneous space under $\widetilde{\LsG}$, the loop period domain is a reductive space, thanks to the decomposition $\widetilde{\Lambda_\sigma \gg} = \tilde{\kk} \oplus \Lambda^0_\sigma \gg$ of its Lie algebra. For such spaces, one can define a \emph{canonical connection} for the principal $\tilde{K}$-bundle $\widetilde{\Lambda_\sigma} G \rightarrow \Dd$. Let $\omega_{MC}$ be the Maurer-Cartan form of $\widetilde{\Lambda_\sigma G}$. Then, the curvature of the canonical connection is given by  
\begin{equation} \label{cann_conn}
\Omega = -\frac12 [\omega_{MC}^{\Lambda^0_\sigma \gg},\omega_{MC}^{\Lambda^0_\sigma \gg}]^{\tilde{\kk}},
\end{equation}
where the Lie algebras in exponents stand for the projection on these Lie algebras, relative to the reductive decomposition given above. (see \cite{CaMuP}, equation 12.3.2).

The determinant line bundle over $\Dd$ (resp. $\check{\Dd}$) is endowed with a $\widetilde{\Lambda_\sigma G}$-invariant metric since the character of $\tilde{K}$ which induces this bundle is unitary. We write $\nabla$ for the connection on the determinant line bundle induced by the canonical connection. Then, by \cite{CaMuP}, theorem 12.3.5,

\begin{prop}
The determinant line bundle over $\Dd$ admits a unique structure of holomorphic line bundle such that $\nabla$ is the Chern connection.
\end{prop}

We can compute the curvature of the connection $\nabla$ and in particular prove the following result:

\begin{prop}\label{curv_det}
The connection $\nabla$ has positive curvature in the horizontal directions.
\end{prop}

\begin{proof}
 Since the determinant line bundle and the connection $\nabla$ are invariant under $\widetilde{\Lambda_\sigma G}$, we only consider the base-point of $\Dd$. Over this point, the curvature of $\nabla^\Dd$ is a skew-symmetric map $\beta$ from $\Lambda^0_\sigma \gg \otimes \Lambda^0_\sigma \gg$ with values in $\End(\C)$. From equation \eqref{cann_conn}, it is given by
$$
\beta(\xi,\eta)w = -\frac12 [\xi,\eta]^{\tilde{\kk}} \cdot w,
$$
where $w \in \C$, $\xi, \eta \in \Lambda^0_\sigma \gg$ and the dot is the infinitesimal action of the character of $\tilde{K}$ defined above. In the decomposition $\tilde{\kk} = \kk \oplus i\R$, only the second component acts non-trivially on $\C$. From the description of the Lie algebra $\tilde{\Lambda \kk}$, this gives:
\begin{align}\label{equa_curv}
\beta(\xi,\eta)w = -\frac 12 \omega(\xi,\eta)w.
\end{align}

We write $I$ for the almost-complex structure on $\Lambda^0_\sigma \gg$. If $f$ is in $\Lambda^0_\sigma \gg$, we write
$$
f(\lambda) = \sum_{n \neq 0} a_n \lambda^n.
$$
We recall that $a_n = (-1)^{n+1} a_n^\ast$ and that $I.f$ is given by 
$$
(I.f)(\lambda) = i\sum_{n < 0} a_n \lambda^n - i \sum_{n > 0} a_n \lambda^n.
$$
Hence, as a function of $\theta$, the derivative of $I.f$ is given by
$$
(I.f)'(\lambda) = -\sum_{n < 0} n a_n \lambda^n + \sum_{n > 0} n a_n \lambda^n.
$$
Since $<X,Y> = -\Tr(XY)$, inverting the sum and the integral, we get:
\begin{align*}
i\beta(f,I.f) &= -\frac{2\pi}{4\pi} (\sum_{n < 0} -\Tr(n a_n a_{-n}) + \sum_{n > 0} - \Tr(-n a_n a_{-n})) \\
&= \sum_{n < 0} n (-1)^{n+1} \Tr(a_n a_n^\ast)
\end{align*}

Since a function in the horizontal distribution has only Fourier coefficients in degrees $1$ and $-1$, this proves the proposition.

\end{proof}

\subsection{Relation with the Hitchin energy} \label{hitch_energy}

In this subsection, we establish a relation between a certain $(1,1)$-form on $X$, constructed from a harmonic bundle, and the cohomology of the period domain. This leads us to an integrality result for a cohomology class and to a new proof of a particular known case of the Carlson-Toledo conjecture.

\subsubsection{A fundamental relation}

Let $(\Ee,D,h)$ be a harmonic bundle over a complex manifold $X$ whose associated Higgs field is denoted by $\theta$.

\begin{defn}
 The \emph{Hitchin energy} is the $(1,1)$-form $\beta_X := \frac {1}{4i\pi} \Tr(\theta \wedge \theta^\ast)$. Here, $\Tr$ is the trace operator on $\End(\Ee)$ and $\theta^\ast$ is the adjoint of $\theta$ for the harmonic metric $h$.
\end{defn}

\begin{prop}
 The Hitchin energy $\beta_X$ is closed.
\end{prop}

\begin{proof}
 We recall the following: if $\Ee$ is a vector bundle over $X$ and $\alpha$ is a $k$-form with values in $\End(\Ee)$, then for any connection $\nabla$ on $\Ee$:
$$
d\Tr(\alpha) = \Tr(\nabla \alpha).
$$
Using this formula, the differential of the Hitchin energy is given by 
$$
4i\pi d\beta_X = \Tr(\nabla \theta \wedge \theta^\ast - \theta \wedge \nabla\theta^\ast),
$$
where $\nabla = D - (\theta + \theta^\ast)$ is the canonical metric connection of the harmonic bundle $\Ee$. By (the proof of) lemma \ref{Dlamb_harm}, $\nabla \theta = \nabla \theta^\ast = 0$. This concludes the proof.
\end{proof}

We write $\beta_\Dd$ for the curvature of the determinant line bundle. Let $\beta_{\tilde{X}}$ be the pullback of the Hitchin energy to $\tilde{X}$ and let $f: \tilde{X} \rightarrow \Dd$ be the period map.

\begin{prop}\label{rel_beta}
The relation $\beta_{\tilde{X}} = \frac{i}{2\pi} f^\ast(\beta_\Dd)$ holds.
\end{prop}

\begin{proof}
 We use the notations of subsection \ref{per_map}. We have defined an isomorphism $\gamma_x : \Ee_x \rightarrow H^0_{f(x)}$. This induces an isomorphism from $\Lambda^{-1,1}_\sigma \End(\Ee_x)$ to $\Lambda^{-1,1}_{\sigma,f(x)} \gg$. Under this isomorphism, the $1$-form $\alpha = \lambda^{-1}\theta + \lambda\theta^\ast$ corresponds to the $1$-form $\lambda^{-1}\d f + \lambda \bar{\d} f$. If $X$ and $Y$ are vector fields of type $(1,0)$ and $(0,1)$, we thus get
$$
\Tr(\theta(X)\theta^\ast(Y))_x = \Tr(d_xf(X) d_x f(Y))_{f(x)},
$$
since the trace of two conjugated endomorphisms is the same. On the other hand, by definition \ref{cocycle} and equation \eqref{equa_curv}, the curvature $\beta$ is given at the base-point $o$ of $\Dd$ by
$$
\beta_{\Dd,o}(X,Y) = -\frac 12 \Tr(XY),
$$
where $X$ is of type $(1,0)$ with values in $\Lambda^{-1} \gg$ and $Y$ of type $(0,1)$ with values in $\Lambda^1 \gg$. Here, $X$ and $Y$ are seen as endomorphisms on the vector space $\C^n$, which is isomorphic to $H^0_{o}$. Since the curvature is invariant under $\LsG$, the formula
$$
\beta_{\Dd,y}(X,Y) = -\frac 12 \Tr(XY)
$$
holds, where $X$ and $Y$ are now with values in $\Lambda^{-1}_y \gg$ and $\Lambda^{1}_y \gg$ and are seen as endomorphism of $H^0_y$. We thus get that
$$
\Tr(\theta(X)\theta^\ast(Y))_x = -2 \beta_{\Dd,f(x)}(d_xf(X),d_xf(Y)).
$$
In other terms, $\beta_{\tilde{X}} = \frac i{2\pi} f^\ast(\beta_\Dd)$; this concludes the proof.
\end{proof}

An immediate consequence of this fact is the following corollary, which seems unknown in the literature.

\begin{cor}\label{int_class}
 The cohomology class of $\beta_{\tilde{X}}$ lives in $H^2(\tilde{X},\Z)$.
\end{cor}

\begin{rem}
 There is \emph{a priori} no reason that the cohomology class of $\beta_X$ lives in $H^2(X,\Z)$. For instance, let $X$ be an elliptic curve and let $\omega$ be a non-zero holomorphic $1$-form on $X$. Let $E$ be the trivial holomorphic line bundle over $X$. Then $(E,t\omega)$ is a Higgs bundle for any $t \in \C$ and we check easily that the trivial Hermitian metric on $E$ is a harmonic metric for any $t \in \C$. The \emph{Hitchin energy} of $(E,t\omega)$ is given by $\frac{|t|^2}{4i\pi} \omega \wedge \bar{\omega}$; since its cohomology class does not vanish, it is not integral for almost all $t \in \C$.

 Writing $L_{\tilde{X}}$ for the pullback of the determinant line bundle to $\tilde{X}$, $\beta_X$ has integral cohomology class if and only if $L_{\tilde{X}}$ is the pullback of a holomorphic line bundle $L_X$. This holomorphic line bundle, having curvature $\beta_X$, would be positive when the period map is an immersion, thanks to propositions \ref{curv_det} and \ref{rel_beta} and the horizontality of the period map. In particular, $X$ will be a projective variety when it is compact.

An integrality criterion is given in theorem \ref{crit_int}.

%This idea works in the case of variations of Hodge structures. In \cite{GrIII}, Griffiths defines in section $7$ a \emph{canonical bundle} on $X$, whose curvature is $\beta_X$ with our notations, and in Theorem $9.7$, he uses 
\end{rem}

\subsubsection{Carlson-Toledo conjecture}

An interesting conjecture concerning fundamental groups of compact K\"ahler manifolds is the following:

\begin{conj}
 Let $\Gamma$ be an infinite group, which is the fundamental group of a compact K\"ahler manifold. Then, virtually, $H^2(\Gamma,\R) \neq 0$.
\end{conj}

The following theorem \ref{CT} is a particular case of lemma 3.2 in \cite{KMK}. Our proof here is quite different and relies on corollary \ref{int_class}. In both proofs, the assumption that the fundamental group admits a non-rigid representation is crucial.

\begin{defn}
 Let $\Gamma$ be a finitely generated group. An irreducible representation $\rho$ in $R^s(\Gamma,G)$ is \emph{rigid} if its connected component in $R^s(\Gamma,G)$ coincides with its conjugation class under $G$.
\end{defn}

\begin{thm}\label{CT}
 Let $X$ be a compact K\"ahler manifold such that there exists a non-rigid irreducible representation of $\pi_1(X)$ in $G$. Then, $H^2(\pi_1(X),\R)$ is non-trivial.
\end{thm}

\begin{proof}
Let $\rho$ be an irreducible representation of $\pi_1(X)$. Let $\Lambda$ be the connected component of $\rho$ in $R^s(\pi_1(X),G)$. We write $\beta_{X,\lambda}$ for the Hitchin energy of the representation $\lambda$ in $\Lambda$. It varies continuously with $\lambda$. Moreover, by corollary \ref{int_class}, its pullback $\beta_{\tilde{X},\lambda}$ to $\tilde{X}$ has an integral cohomology class. The cohomology class of $\beta_{\tilde{X},\lambda}$ is thus constant in $\lambda$.

Hence, if $\lambda$ and $\lambda'$ are in $\Lambda$, the cohomology class of $\beta_{X,\lambda} - \beta_{X,\lambda'}$ is in the kernel of the pullback map $H^2(X,\R) \rightarrow H^2(\tilde{X},\R)^{\pi_1(X)}$. This map is part of the following exact sequence
$$
0 \rightarrow H^2(\pi_1(X),\R) \rightarrow H^2(X,\R) \rightarrow H^2(\tilde{X},\R)^{\pi_1(X)},
$$
and the cohomology class of $\beta_{X,\lambda} - \beta_{X,\lambda'}$ thus lives in $H^2(\Gamma,\R)$. We can assume that this cohomology class is always zero (or the proof is over). Then, the function $F$ defined before theorem \ref{VHS_ptcrit} is constant in $\Lambda/G$, which is an open set in $X^s(\pi_1(X),G)$. All smooth points in $\Lambda/G$ are thus critical points of $F$, hence come from variations of Hodge structures by theorem \ref{VHS_ptcrit}. The set of smooth points in $\Lambda/G$ is thus contained in a totally real analytic submanifold of $X^s(\pi_1(X),G)$ by lemma \ref{VHS_totreal}. Since it is also open, it has to be a point; which is possible only if $\Lambda/G$ is itself a point. Hence, $\rho$ is a rigid representation.

\end{proof}

\subsubsection{Criterion of integrality}

Let $(\Ee,D,h)$ be a harmonic bundle over a complex manifold $X$ and let $f: \tilde{X} \rightarrow \Dd$ and $\rho_{\tot}: \pi_1(X) \rightarrow \Lambda_\sigma G$ be the period map and total monodromy. Thanks to the monodromy $\rho_{\tot}$, we can define a central $S^1$-extension $\widetilde{\pi_1(X)}$ of $\pi_1(X)$ from the central $S^1$-extension $\widetilde{\LsG}$ of $\LsG$ by the following formula:
$$
\widetilde{\pi_1(X)} = \{(\gamma,\tilde{g}) \in \pi_1(X) \times \widetilde{\Lambda_\sigma G} \mid \rho_{\tot}(\gamma) = p(\tilde{g})\},
$$
where $p: \widetilde{\Lambda_\sigma G} \rightarrow \Lambda_\sigma G$ is the projection map.

We recall that, if $A$ is an abelian group and $\Gamma$ a discrete group, then the central extensions of $\Gamma$ by $A$ are classified up to isomorphism by the elements of $H^2(\Gamma, A)$, where $A$ is seen as a trivial $\Gamma$-module. We write $e \in H^2(\pi_1(X),S^1)$ for the cohomology class of $\widetilde{\pi_1(X)}$.

\begin{thm}\label{crit_int}
 The cohomology class of $\beta_X$ in $H^2(X,\R)$ is integral if and only if $e = 0$ in $H^2(\pi_1(X),S^1)$.
\end{thm}

\begin{proof}
 We write $L \rightarrow \Dd$ for the determinant line bundle over $\Dd$, $f: \tilde{X} \rightarrow \Dd$ for the period map and $\pi: \tilde{X} \rightarrow X$ for the canonical projection map.

 Assume that $e = 0$. Then, as a $S^1$-central extension, $\widetilde{\pi_1(X)}$ is isomorphic to $\pi_1(X) \times S^1$. In particular, the bundle $f^\ast L$ is then $\pi_1(X)$-equivariant. Hence, it defines a holomorphic line bundle $L_X$ over $X$ such that $\pi^\ast L_X \cong f^\ast L$. The bundle $f^\ast L$ is equipped with a connection whose Chern form is $\beta_{\tilde{X}}$ by proposition \ref{rel_beta}. Since this connection is $\pi_1(X)$-invariant, the bundle $L_X$ is also equipped with a connection and its Chern form $\omega_X$ satisfies $\pi^\ast(\omega_X) = \beta_{\tilde{X}} = \pi^\ast(\beta_X)$. Since, $\pi$ is a local diffeomorphism, this implies that $\omega_X = \beta_X$. This shows that the cohomology class of $\beta_X$ is integral since $\omega_X$ is the Chern form of a line bundle.

 Conversely, we assume that $\beta_X$ has integral cohomology class. Then, the bundle $f^\ast L \rightarrow \tilde{X}$ is isomorphic to the pullback by $\pi$ of a bundle $L_X \rightarrow X$. In particular, the bundle $f^\ast L$ is then $\pi_1(X)$-equivariant and the trivial $S^1$-central extension $\pi_1(X) \times S^1$ can thus be identified with the group of vector bundle isometries of $f^\ast L \rightarrow \tilde{X}$ that cover the action of an element of $\pi_1(X)$. On the other hand, an element of $\widetilde{\pi_1(X)}$ also acts on this bundle as a vector bundle isometry covering the action of an element of $\pi_1(X)$. Since both groups $\pi_1(X) \times S^1$ and $\widetilde{\pi_1(X)}$ are central $S^1$-extensions of $\pi_1(X)$, they are isomorphic as $S^1$-extensions of $\pi_1(X)$.
\end{proof}

\begin{rem}
 The condition $e=0$ means that the total monodromy $\rho_{\tot}: \pi_1(X) \rightarrow \Lambda_\sigma G$ can be lifted to $\widetilde{\LsG}$. I do not know if such a condition can be studied in some cases.
\end{rem}

\addtocontents{toc}{\protect\setcounter{tocdepth}{1}}
\appendix
\section{Algebra and geometry in infinite dimension} \label{inf_dim}

In this appendix, we collect some results on Hilbert spaces and geometry in infinite dimension.

\subsection{Hilbert spaces}

Let $V$ be a complex vector space of finite dimension and let $H$ denote the Hilbert space $H := L^2(S^1,V)$. Let $T:H \rightarrow H$ be the right-shift operator, defined by $(Tf)(\lambda) = \lambda f(\lambda)$, for $\lambda$ in $S^1$. We write $\End(H)$ for the Banach space of continuous endomorphisms of $H$ and $GL(H)$ for its group of invertible elements. We can consider the Banach space $L^\infty(S^1, \End(V))$ as a subspace of $\End(H)$, thanks to the continuous injection $L^\infty(S^1,\End(V)) \hookrightarrow \End(H)$, given by $(M.f)(\lambda) = M(\lambda).f(\lambda)$, where the dot in the right hand side is the evaluation map. The group of invertible elements of $L^\infty(S^1,\End(V))$ is denoted $\Lambda^\infty GL(V)$; it is the set of (equivalence classes of) measurable maps $\gamma: S^1 \rightarrow GL(V)$ such that both $\gamma$ and $\gamma^{-1}$ are essentially bounded.

Next proposition is theorem (6.1.1) of \cite{PS} for the group statement:

\begin{prop}\label{comm_shift}
 The Banach space $L^\infty(S^1,\End(V))$ can be identified to the commutant of the right-shift operator $T$ in $\End(H)$. The Banach-Lie group $\Lambda^\infty GL(V)$ is the commutant of $T$ in the Banach-Lie group $GL(H)$.
\end{prop}

We will need a slightly more general statement. Let $V_1$ and $V_2$ be two complex vector spaces and write $H_i = L^2(S^1,V_i)$. Then, with obvious notations:
\begin{prop}\label{comm_shift2}
 The Banach space $L^\infty(S^1,\Hom(V_1,V_2))$ can be identified with the subspace of $\Hom(H_1,H_2)$ of linear maps that intertwine the right-shift operators $T_1$ and $T_2$.
\end{prop}

Let $h$ be a Hermitian inner product on $V$. We consider the Hermitian form $B$ on $H$, defined by:
$$ B(f,g) = \int_{S^1} h(f(\lambda),g(-\lambda))d\nu(\lambda),$$
where $\nu$ is the invariant volume form of mass $1$ on $S^1,$ and $f$ and $g$ are in $H$. Let $\sigma: GL(V) \rightarrow GL(V)$ denote the Cartan involution, relatively to the unitary group of $h$. We denote by $L^\infty_\sigma(S^1,GL(V))$ the subgroup of maps $\gamma: S^1\rightarrow GL(V)$, satisfying the equation $\sigma(\gamma(\lambda)) = \gamma(-\lambda)$, for almost all $\lambda$ in $S^1$. 

We write $\hat{\sigma}: \Lambda^\infty GL(V) \rightarrow \Lambda^\infty GL(V)$ for the involution defined by $\hat{\sigma}(\gamma)(\lambda) = \sigma(\gamma(-\lambda))$, so that $\Lambda^\infty_\sigma GL(V)$ is the subgroup of loops which are fixed by $\hat{\sigma}$.

\begin{prop}\label{B_isom}
 The group $\Lambda^\infty_\sigma GL(V)$ is the subgroup of $\Lambda^\infty GL(V)$ that acts $B$-isometrically on $H$.
\end{prop}

\begin{proof}
Let $f,g$ be in $H$ and let $\gamma$ in $\Lambda^\infty GL(V)$. Then
\begin{align*}
B(\gamma.f, \gamma.g) &= \int_{S^1} h(\gamma(\lambda)f(\lambda), \gamma(-\lambda)g(-\lambda)) d\nu(\lambda) \\
&= \int_{S^1} h(f(\lambda), \sigma(\gamma(\lambda))^{-1}\gamma(-\lambda) g(-\lambda)) d\nu(\lambda) \\
&= h(f, (\hat{\sigma}(\gamma)^{-1}\gamma).g).
\end{align*}
Since $B$ is non-degenerate, this concludes the proof.
\end{proof}

\subsection{Infinite dimensional bundles}

We assume that the reader is familiar with differential geometry in infinite dimension. A standard reference is \cite{Lan}. First, we reinterpret proposition \ref{comm_shift}, in a geometrical setting. 

Let $\Ee$ be a complex vector bundle of finite dimension on a (finite dimensional) differentiable manifold $X$. Let $D$ be an arbitrary connection on the vector bundle $\Ee$; it induces a connection $\tilde{D}$ on the Hilbert bundle $\Hha = L^2(S^1,\Ee)$ by the formula
$$
(\tilde{D}_Y f)(\lambda,v) = (D_Y f(\lambda,\cdot))(v),
$$
where $f$ is a local section of $\Hha$, $Y$ is a local vector field on $X$, $\lambda$ is in $S^1$ and $v$ is in $\Ee$.

 The connections $D$ and $\tilde{D}$ give an origin to the space of connections of $\Ee$ and $\Hha$. Hence, this identifies these spaces to $\Gamma(X,T^\ast X \otimes \End(\Ee))$ and to $\Gamma(X,T^\ast X \otimes \End(\Hha))$. There is a natural embedding of Banach bundles $L^\infty(S^1,\End(\Ee)) \rightarrow \End(\Hha))$.

\begin{cor}\label{conn_T}
 A connection $\tilde{\nabla}$ on $\Hha$ lives in the affine subspace $\tilde{D} + \Gamma(X,T^\ast X \otimes L^\infty(S^1,\End(\Ee)))$ if and only if the right-shift operator $\Tt$ is parallel for $\tilde{\nabla}$.
\end{cor}

Hence, if $(\nabla_\lambda)_{\lambda \in S^1}$ is a circle of connections on $\Ee$ (with weak conditions on the regularity with respect to $\lambda$), one can define a connection $\tilde{\nabla}$ on $\Hha$ by the formula
$$
(\tilde{\nabla}_Y f)(\lambda,v) = ((\nabla_\lambda)_Y f(\lambda,\cdot))(v),
$$
where $f$ is a local section of $\Hha$, $Y$ is a local vector field on $X$, $\lambda$ is in $S^1$ and $v$ is in $\Ee$. Conversely, a connection $\tilde{\nabla}$ on $\Hha$ is obtained in this way if and only if the right-shift operator $\Tt$ is parallel for $\tilde{\nabla}$.\\

We also give a Riemann-Hilbert correspondence for Banach bundles with flat connections. The proof is identical to the finite-dimensional setting since one can consider parallel transport in the same way.

Let $X$ be a (finite-dimensional) connected differentiable manifold. Let $x_0$ be a base-point in $X$ and let $p:\tilde{X} \rightarrow X$ be the universal cover of $X$. Let $E$ be a complex Banach space.

\begin{prop}
 If $\rho: \pi_1(X,x_0) \rightarrow \Aut(E)$ is a representation of $\pi_1(X,x_0)$ in the space of bounded automorphisms of $E$, one defines a Banach bundle on $X$ by
 $$ \Ee_\rho := \tilde{X} \times_\rho E.$$
 It is endowed with a flat connection $D$, which comes from the canonical connection $d$ on the trivial bundle $\tilde{X} \times E \rightarrow \tilde{X}$.

Conversely, if $(\Ee,D)$ is a Banach bundle modeled on $E$ with flat connection over $X$, then the parallel transport defines a \emph{monodromy representation} $\rho: \pi_1(X,x_0) \rightarrow \Aut(\Ee_{x_0}) \cong \Aut(E)$.

These two constructions give an equivalence of categories between the category of representations of $\pi_1(X,x_0)$ in the bounded automorphisms of $E$ and the category of Banach bundles modeled on $E$ endowed with a flat connection over $X$.
\end{prop}

\section{Moduli spaces} \label{moduli}

In this appendix, we collect some facts about the various moduli spaces that one can attach to a compact K\"ahler manifold. This is only used in the proof of theorem \ref{CT}.

\subsection{Betti moduli space} \label{Betti}

\subsubsection{Construction}

Let $\Gamma$ be a finitely presented group and let $G = GL(n,\C)$. We briefly review the construction of the moduli space of representations $X(\Gamma,G)$, which is a \emph{categorical quotient} (in the sense of geometric invariant theory) for the action of $G$ by conjugation on the space of representations $R(\Gamma,G)$. All this material is contained in \cite{JohnMill}, (though in the case $G$ semisimple rather than reductive).

Let $\gamma_1,\dots,\gamma_N$ be generators of $\Gamma$ and let $r_1,\dots,r_k$ be relators that define a presentation of $G$. We define $R(\Gamma,G)$ to be the subset of $G^N$ of elements $(M_1,\dots,M_N)$ satisfying the relations $r_i(M_1,\dots,M_N) = 0$. In this way, $R(\Gamma,G)$ is realized as an affine algebraic variety and this affine algebraic structure is independent of the presentation of $G$. The quotient variety $X(\Gamma,G)$ is by definition the affine variety corresponding to the algebra of $G$-invariant polynomials on $R(\Gamma,G)$. We write $\pi: R(\Gamma,G) \rightarrow X(\Gamma,G)$. The map $\pi$ does not in general induce a bijection $R(\Gamma,G)/G \cong X(\Gamma,G)$.

If $U$ is a real form of $G$, both $R(\Gamma,G)$ and $X(\Gamma,G)$ are defined over the real numbers. Moreover, one can define in the same way a space $X(\Gamma,U)$ and this space is contained in the real points of $X(\Gamma,G)$.

A representation $\rho$ in $R(\Gamma,G)$ is \emph{stable} if its orbit under conjugation is closed in $R(\Gamma,G)$. It is equivalent to the condition that $\Im \rho$ is not contained in any proper parabolic subgroup of $G$; since $G = GL(n,\C)$, this simply means that the representation is irreducible. 

We write $R^s(\Gamma,G)$ for the set of stable representations. It is Zariski-open in $R(\Gamma,G)$. Its image by $\pi$ is written $X^s(\Gamma,G)$; it is Zariski-open in $X(\Gamma,G)$. If $\Gamma$ is the fundamental group of some topological space $X$, we write $\Mm_B(X)$ for $X^s(\Gamma,G)$ and call it the \emph{Betti moduli space} of $X$. The map $\pi: R^s(\Gamma,G) \rightarrow X^s(\Gamma,G)$ induces a homeomorphism $R^s(\Gamma,G)/G \cong X^s(\Gamma,G)$. If $U$ is a real form of $G$, we define in the same way $X^s(\Gamma,U)$. It is homeomorphic to $(R^s(\Gamma,G) \cap R(\Gamma,U))/U$ and is a subset of the real points of $X^s(\Gamma,G)$.

\subsubsection{Variations of Hodge structures}
We write $Z$ for the equivalence classes in $X^s(\Gamma,G)$ of representations whose associated flat bundles admit a complex variation of Hodge structures. We will need the following lemma:

\begin{lem}\label{VHS_totreal}
 The subset $Z$ is contained in a totally real analytic subspace of the complex-analytic space $X^s(\Gamma,G)$
\end{lem}

\begin{proof}
 If $\rho$ is a representation coming from a complex variation of Hodge structures, then it is conjugated to a representation with values in a real form $U(p,q)$ of $G$. Hence, $Z$ is contained in the union of the sets $X^s(\Gamma,U(p,q))$. Each of these sets is included in the set of real points (for different real structures) of $X^s(\Gamma,G)$; in particular, each one is included in a totally real analytic subset of the complex-analytic space $X^s(\Gamma,G)$. The union being finite, this concludes the proof.
\end{proof}

\subsection{Hyperk\"ahler structure}

In this subsection, we recall the construction of a hyperk\"ahler structure on (the smooth points of) the Betti moduli space of $X$. This is well known for smooth projective manifolds: this reduces to the case of a complex curve, which is treated in \cite{Hitch}. This has also been done for a compact K\"ahler manifold $X$ in \cite{Fuj}, which we follow.

\subsubsection{Space of connections}
Let $M$ be a compact K\"ahler manifold, with K\"ahler form $\omega$. Let $(\Ee,h)$ be a Hermitian vector bundle on $X$. We choose an integer $k \geq \dim M + 2$ and consider the spaces $\Aa$ (resp $\Aa_\R$) of connections with Sobolev regularity $H_k$ on $\Ee$ (resp. metric connections on $(\Ee,h)$). The space $\Aa$ has a natural structure of a hyperk\"ahler affine Hilbert space defined in the following way:

We choose a metric connection $D_0$ so that we identifify $\Aa$ (resp. $\Aa_\R$) to the Hilbert spaces $H^1_k(M,\End(\Ee))$ (resp. $H^1_k(M,\uu(\Ee,h))$): these notations stand for the space of global $1$-forms with regularity $H_k$ in the bundles $\End(\Ee)$ and $\uu(\Ee,h)$. From the complex structure on $M$, the space $\Aa_\R = H^1_k(M,\uu(\Ee,h))$ inherits a complex structure $C$. It is also endowed with a $C$-invariant inner product $<\cdot,\cdot>$ defined by
$$
<\phi,\psi> = \int_M <\phi(x),\psi(x)>_x \frac{\omega^n}{n!},
$$
where the inner product $<\cdot,\cdot>_x$ on $T_x^\ast M \otimes \uu(\Ee_x,h_x)$ is induced from the K\"ahler metric on $T_x M$ and from \emph{minus} the Killing form on $\uu(\Ee_x,h_x)$.

We consider $\Aa$ as the complexification of $\Aa_\R$ and we write $J$ for the corresponding complex structure. We extend $C$ and $<\cdot,\cdot>$ by complex linearity to $\Aa$. Then $h(x,y) := <x,\bar{y}>$ (the bar is relative to the real form $\Aa_\R$) is a Hermitian metric on $\Aa$; we write $g$ for its real part. Finally, we write $I$ for the unique antilinear extension of $C$ to $\Aa$. Then, (\cite{Fuj}, (4.1)):

\begin{lem}\label{lem_hyperk}
 The space $(\Aa,g,I,J)$ is hyperk\"ahler.
\end{lem}

There is also a natural $S^1$-action on $\Aa$ (\cite{Fuj}, (4.2)). Let $\lambda$ be a unit number in $\C^\ast$ and let $D$ be an element in $\Aa$. As usual, we write $D = \nabla + \alpha$, where $\nabla$ is a metric connection and $\alpha$ is a Hermitian $1$-form. We decompose $\alpha$ in types : $\alpha = \alpha^{1,0} + \alpha^{0,1}$. The action of $S^1$ on $\Aa$ is given by
\begin{equation}\label{circle_action}
\lambda.D = \nabla + \lambda^{-2} \alpha^{1,0} + \lambda^2 \alpha^{0,1}.
\end{equation}

\subsubsection{Hyperk\"ahler reduction}

A connection $D$ in $\Aa$ is \emph{Einstein} if its curvature $F_D$ satisfies
$$
i \Tr(F_D) = \lambda \Id_{\Ee},
$$
where $\lambda$ is some constant depending only on $M$, $\omega$ and $\Ee$ and where the trace operator is defined from $H^{1,1}_k(M,\End(\Ee))$ to $H_k(M,\End(\Ee))$ and is induced by the K\"ahler metric.

Let $D$ be a connection in $\Aa$ and write $D = \nabla + \alpha$. We write $\nabla^\ast$ for the formal adjoint of $\nabla$ and we say that $D$ is \emph{weakly harmonic} if the equation
$$
\nabla^\ast \alpha = 0
$$
is satisfied.

A connection $D$ is \emph{irreducible} if there is no non-trivial subbundle $F$ of $\Ee$, which is stable by the connection.\\

A remarkable point is that both the Einstein and weakly harmonic conditions can be interpreted as asking that the connections are in the zero sets of moment maps, relatively to the hyperk\"ahler structure on $\Aa$ and the action of the complex gauge group $\Gg := H_{k+1}(M,GL(\Ee))$ and its real form $\Kk := H_{k+1}(M,U(\Ee,h))$ (\cite{Fuj}, (6.2) and (6.3)). We write $\tilde{\Ff}$ for the subset of $\Aa$ of irreducible, Einstein and weakly harmonic connections. Then hyperk\"ahler reduction leads to:

\begin{thm}[\cite{Fuj}, (6.6.1)]
 We write $\Ff = \tilde{\Ff}/\Kk$ for the set of equivalence classes of irreducible, Einstein and weakly harmonic connections on $(\Ee,h)$. Then, $\Ff$ inherits from $\Aa$ a structure of hyperk\"ahler Hilbert orbifold and a $S^1$-action.
\end{thm}

\subsubsection{Hyperk\"ahler structure on the moduli space}

We now assume that the first and second Chern classes of $\Ee$ vanish. We write $\Nn$ for the set of isomorphism classes of connections $D$ that are irreducible, flat and admit an harmonic metric. Then, (\cite{Fuj}, (8.1.2)):
\begin{lem}
 The space $\Nn$ can be naturally realized as a finite-dimensional locally closed complex analytic subspace of $\Ff$.
\end{lem}

In subsection \ref{Betti}, we have defined the Betti moduli space $\Mm_B(X)$ of irreducible representations from $\pi_1(X)$ to $GL(n,\C)$. We consider its subset $\Mm_{B,\Ee}(X)$ of representations whose associated bundle is isomorphic to $\Ee$ as a smooth vector bundle. It is a union of connected components of $\Mm_B$. By the Corlette-Donaldson theorem \cite{Sim_higloc}, the spaces $\Nn$ and $\Mm_{B,\Ee}$ are in bijection. Moreover,

\begin{lem}[\cite{Fuj}, (8.2.3)]
 The complex analytic structures of $\Nn$ and $\Mm_{B,\Ee}$ are the same.
\end{lem}

We finally write $\Mm_0$ for the set of smooth points in the underlying reduced space of $\Mm_{B,\Ee}$.

\begin{thm}[\cite{Fuj}, (8.3.1)]
By the inclusion $\Mm_0 \hookrightarrow \Ff$, $\Mm_0$ inherits a hyperk\"ahler structure. Moreover, $\Mm_0$ is stable by the action of $S^1$.
\end{thm}

\begin{rem}
 Much more is said in \cite{Fuj}. Our aim here is only to define the hyperk\"ahler structure on the Betti moduli space of $X$. If one is interested in Higgs bundles, a moduli space of stable Higgs bundles can be defined (for $X$ projective, this is done in \cite{Simp3} and \cite{Simp4}) and its set of smooth points will be in bijection with $\Mm_0$. Moreover, its complex structure will be one of the complex structures of the hyperk\"ahler manifold $\Mm_0$. We omit the details and refer the interested reader to \cite{Fuj}.
\end{rem}

\subsubsection{Variations of Hodge structures}

The fixed points of the circle action on $\Mm_0$ are the irreducible representations whose associated flat bundle admits a complex variation of Hodge structures. This comes from equation \eqref{circle_action} and the well-known corollary 4.2 of \cite{Sim_higloc}, which is our theorem \ref{class_nonab}, in the context of variations of loop Hodge structures. In the space $\Aa$, we consider the functional
$$
F(D) = i\int_M \Tr(\phi^{1,0} \wedge \phi^{0,1}) \wedge \omega^{n-1},
$$ 
where as usual, we write $D = \nabla + \phi$. There exists a real constant $C$ such that $F(D) = C.g(\phi,\phi)$, where $g$ is the hyperk\"ahler metric of $\Aa$.

The metric is invariant under the unitary gauge group $\Kk$, hence $F$ defines a function on $\Ff$, and in particular on $\Mm_0$. We still denote this function by $F$. The tangent space of $\Aa$ at the point $D$ can be canonically identified with $H^1_k(M,\uu(\Ee,h)) \times H^1_k(M,i\uu(\Ee,h))$. Hence, the differential of $F$ is given by
$$
d_D F(\chi,\psi) = 2Cg(\phi,\psi).
$$
On the other hand, the infinitesimal vector field $Y$ generating the circle action on $\Aa$ is given by $Y_D = -2i \phi^{1,0} + 2i \phi^{0,1} = 2I\phi$, where the complex structure $I$ was defined before lemma \ref{lem_hyperk}. Hence, the gradient of $F$ is given by $C.I Y$.

Since the hyperk\"ahler structure of $\Aa$ induces a hyperk\"ahler structure on $\Mm_0$, the complex structure $I$ is in particular well-defined on $\Mm_0$ and we get the following theorem:
\begin{thm}\label{VHS_ptcrit}
 The critical values of $F$ on $\Mm_0$ are the irreducible representations whose associated flat bundle admits a complex variation of Hodge structures.
\end{thm}

\begin{rem}
 A more precise version of this theorem has been obtained in \cite{Spina} by a different method.
\end{rem}

\section{Non-polarized variations of loop Hodge structures} \label{non_pol}

In this appendix, we generalize the notion of variation of loop Hodge structures, when there is no polarization. This corresponds to the notion of \emph{non-polarized harmonic bundle}, defined in \cite{Sim_twis}

\subsection{Algebraic formulation}

Theorem \ref{thm_fund} and its proof show that the Hilbert bundles $\Kk$ that we consider in the definition of a variation of loop Hodge structures are obtained from a finite-dimensional bundle $\Ee$ \emph{via} the functor $\Ee \mapsto L^2(S^1,\Ee)$ and that the flat connection $D$ on $\Kk$ involves only a finite number of exponents of $\lambda$, relatively to the isomorphism $\Kk \cong L^2(S^1,\Ee)$. Hence, it is natural to give a more algebraic definition of variations of loop Hodge structures.

\begin{defn}
Let $\Ee$ be a finite-dimensional complex vector bundle on a differentiable manifold $X$. We define $\Ee[t,t^{-1}]$ to be the infinite-dimensional complex vector bundle on $X$ whose sheaf of smooth sections $\mathcal{F}$ is defined for any open set $U$ by $\mathcal{F}(U) = \bigcup_{n \in \N} \mathcal{F}^n(U)$, where
$$
\mathcal{F}^n(U) := \{ \sum_{|k| \leq n} a_k t^k \,|\, a_k \in \Gamma(U,\Ee)\}.
$$

We define the \emph{right-shift operator} $T$ on $\Ee[t,t^{-1}]$ to be the multiplication by the formal variable $t$.

If $(\Ee,h)$ is a Hermitian vector bundle, we define the \emph{Krein metric} $\Bb$ on $\Ee[t,t^{-1}]$ by
$$ \Bb(\sum_{n \in \Z} f_n t^n,\sum_{n \in \Z} g_n t^n) = \sum_{n \in \Z} (-1)^n h(f_n,g_n),$$
where all the sums are in fact finite.
\end{defn}

\begin{defn}
 A \emph{connection} $D$ on $\Ee[t,t^{-1}]$ is a $\C$-linear map $\Gamma(X,TX \otimes \Ee[t,t^{-1}]) \rightarrow \Gamma(X,\Ee[t,t^{-1}])$ which is $\mathcal{C}^\infty(X)$-linear in the $TX$ direction and which satisfies Leibniz's rule. As usual, one can define the notion of \emph{flatness} of a connection or speak of \emph{flat tensors}.
\end{defn}

\begin{prop}\label{conn_Talg}
 Let $D$ be a connection on $\Ee[t,t^{-1}]$. Then $T$ is $D$-parallel if and only if $D$ can be written $D = D_0 + \sum_{1 \leq |k| \leq n_0} a_k t^k$, where $D_0$ is a connection on $\Ee$ and the $a_k$ are $1$-forms with values in $\End(\Ee)$.
\end{prop}

\begin{proof}
This is analogous to proposition \ref{conn_T} and is in fact easier to prove since there are no topological issues. Let $D_0$ be any connection on $\Ee$ and consider it as a connection $\tilde{D}_0$ on $\Ee[t,t^{-1}]$. Then, $\omega:= D - \tilde{D}_0$ is a section of $T^\ast X \otimes \End(\Ee[t,t^{-1}])$ that commutes with $T$. Considering the action of $\omega$ on elements in $\Ee \subset \Ee[t,t^{-1}]$, one obtains a form $A \in \Gamma(X,T^\ast X \otimes \End(\Ee)[t,t^{-1}])$ such that, for every $x$ in $X$, $Y_x$ in $T_x X$ and $v_x$ in $E_x$, $\omega_x(Y_x)(v_x) = A_x(Y_x)(v_x)$. Since $\omega$ commutes with $T$, this equality is necessarily true for $v$ in $E[t,t^{-1}]_x$.

Otherwise said, $D$ lives in $\tilde{D}_0 + \Gamma(X,T^\ast X \otimes \End(\Ee)[t,t^{-1}])$ and this is exactly the statement of the theorem.
\end{proof}

\begin{defn}
 An \emph{algebraic variation of (complex polarized) loop Hodge structures} over a complex manifold $X$ is the datum of a finite-dimensional Hermitian bundle $(\Ee,h)$ over $X$ and a flat connection $D = \d + \bar{\d}$ on $\Ee[t,t^{-1}]$ 
\begin{comment}
\begin{itemize}
 \item a Hermitian bundle $(V,h)$ over $X$;
 \item a connection $D = \d + \bar{\d}$ on $\Ee[t,t^{-1}]$;
\end{itemize}
 \end{comment}
such that
\begin{enumerate}
 \item the Krein metric $\Bb$ and the right-shift operator $\Tt$ are $D$-flat;
 \item $\bar{\d}: \Ee[t] \rightarrow \Ee[t] \otimes \Aa^{0,1}$;
 \item $D: \Ee[t] \rightarrow t^{-1}\Ee[t] \otimes \Aa^1$.

\end{enumerate}
\end{defn}

\begin{rem}
 By completing the vector bundle $\Ee[t,t^{-1}]$ with respect to its natural pre-Hilbert structure, one obtains a variation of loop Hodge structures in the sense of definition \ref{VNCHS}. On the other hand, by the proof of theorem \ref{thm_fund}, every variation of loop Hodge structures is the completion of an algebraic variation of loop Hodge structures. 

 This algebraic definition may seem simpler than the analytic one. If one is only interested in the variational point of view, this is certainly true. But, as soon as a classifying space for loop Hodge structures is needed, we have to solve ordinary differential equations and we need the Hilbert setting to do so. Moreover, it seems preferable in the definition of loop Hodge structures to emphasize the role of the outgoing subspace $\Ww$ rather than the finite-dimensional subspace $\Ee$ which is the orthocomplement of $\Tt\Ww$ in $\Ww$. This can be done in the Hilbert setting, thanks to proposition \ref{kr_out}; but there is no obvious analogue of this proposition in an algebraic setting. For these reasons, apart from the following discussion of non-polarized variations, we will not use this algebraic setting anymore.
\end{rem}

\subsection{Non-polarized variations}

If one is interested in variations of loop Hodge structures that are not \emph{polarized} (that is, which carry no metric), the algebraic definition looks simpler to generalize in the following way:

\begin{defn}
 A \emph{non-polarized variation of (complex) loop Hodge structures} over a complex manifold $X$ is the datum of a finite-dimensional complex vector bundle $\Ee$ over $X$ and a flat connection $D = \d + \bar{\d}$ on $\Ee[t,t^{-1}]$ such that
\begin{enumerate}
 \item the right-shift operator $T$ is $D$-flat;
 \item $\bar{\d}: \Ee[t] \rightarrow \Ee[t] \otimes \Aa^{0,1}$;
 \item $D: \Ee[t] \rightarrow t^{-1}\Ee[t] \otimes \Aa^1$;
 \item $\d: \Ee[t^{-1}] \rightarrow \Ee[t^{-1}] \otimes \Aa^{1,0}$ \label{diff2np};
 \item $D: \Ee[t^{-1}] \rightarrow t\Ee[t^{-1}] \otimes \Aa^1$ \label{diff1np}.

\end{enumerate}
\end{defn}

\begin{rem}
 In the polarized case, the two additionnal assumptions \ref{diff2np}. and \ref{diff1np}. are satisfied, thanks to the polarization. The interest of this notion lies in the fact that it is equivalent to a non-polarized harmonic bundle, as defined in \cite{Sim_twis}.
\end{rem}

\begin{defn}
 A \emph{non-polarized harmonic bundle} over a complex manifold $X$ is a triple $(\Ee,D',D'')$, where $\Ee$ is a complex vector bundle over $X$, $D'$ (resp. $D''$) is a differential operator on $\Ee$ which satisfy the Leibniz's rule for $\d_X$ (resp. $\bar{\d}_X$); moreover we ask that $D'$ and $D''$ satisfy the integrability conditions $(D')^2 = (D'')^2 = D'D'' + D''D' = 0$.
\end{defn}

\begin{rem}\label{nonpol}
 In the polarized situation, where the flat connection $D$ is written $D = \nabla + \alpha = (\d + \bar{\d}) + (\theta + \theta^\ast)$, one takes $D' = \d + \theta^\ast$ and $D'' = \bar{\d} + \theta$.
\end{rem}

\begin{thm}
 The category of non-polarized variations of loop Hodge structures is equivalent to the category of non-polarized harmonic bundles.
\end{thm}

\begin{proof}[Sketch of the proof]
 The proof is very similar to the proof of theorem \ref{thm_fund}; we only sketch it. We decompose $D' = D'_{1,0} + D'_{0,1}$ and $D'' = D''_{1,0} + D''_{0,1}$ according to their types. For $\lambda$ in $S^1$, a differential operator is defined by
 $$ D_\lambda := D'_{1,0} + D''_{0,1} + \lambda^{-1} D''_{1,0} + \lambda D'_{0,1}.$$
 In the polarized case, this is the usual definition, thanks to remark \ref{nonpol}. For all $\lambda$, $D_\lambda$ is a connection, and one checks that the integrability conditions imply the flatness of $D_\lambda$. This circle of connections thus defines a flat connection in $\Ee[t,t^{-1}]$ (replacing $\lambda$ by the formal variable $t$) and this defines a non-polarized variation of loop Hodge structures.

 Conversely, if $D$ is the flat connection on $\Ee[t,t^{-1}]$, one first uses proposition \ref{conn_Talg} to have a good representation of $D$. Then, the differential assumptions on $\Ee[t]$ and $\Ee[t^{-1}]$ force a lot of vanishings in the coefficients of $D$, as in the proof of theorem \ref{thm_fund}. One obtains that $D$ can be written
 $$ D = D_0 + t^{-1}\alpha_{-1} + t \alpha_1,$$
 where $D_0$ is a connection on $\Ee$, $\alpha_{-1}$ is a $(1,0)$-form with values in $\End(\Ee)$ and $\alpha_1$ is a $(0,1)$-form with values in $\End(\Ee)$.

 The differential operators $D'$ and $D''$ are defined by $D' = D_0^{1,0} + \alpha_1$ and $D'' = D_0^{0,1} + \alpha_{-1}$.
\end{proof}

\begin{rem}
 This jugglery with differential operators appears in particular in \cite{Sim_twis}, lemma 3.1, in the proof of the equivalence between (polarized) harmonic bundles and polarizable pure variations of twistor structures.
\end{rem}

\section{Shafarevich morphism} \label{shafa}

This appendix is a joint work with Yohan Brunebarbe.

\begin{defn}
Let $X$ be a connected compact K\"ahler manifold and let $\rho:\pi_1(X) \rightarrow GL(n,\C)$ be a representation. A \emph{Shafarevich morphism} for $(X,\rho)$ is the datum of a connected complex normal space $\Sh_\rho(X)$ and a holomorphic surjective map with connected fibers $\sh_\rho: X \rightarrow \Sh_\rho(X)$ such that for any connected complex manifold $Z$ and any holomorphic map $f: Z \rightarrow X$, the map $\sh_\rho \circ f$ is constant if and only if the morphism $\rho \circ f_\ast: \pi_1(Z) \rightarrow GL(n,\C)$ has finite image.
\end{defn}

Let $X$ be a compact K\"ahler manifold and let $\rho: \pi_1(X) \rightarrow GL(n,\C)$ be a representation. In theorem 1 of \cite{Eys}, it is shown that if the Zariski-closure of the image of $\rho$ is a semisimple group, then the Shafarevich morphism exists and its image is a projective normal algebraic variety of general type if $\rho(\pi_1(X))$ has no torsion. In theorem \ref{thm_shafa}, we show under simplifying hypotheses that the Shafarevich morphism can be understood \emph{via} the period map attached to the harmonic bundle with monodromy $\rho$. 

We first recall the notion of \emph{Stein factorization} for a proper holomorphic map between complex spaces:

\begin{thm}[Stein factorization, \cite{Gra}, page 213]
 Let $X$ and $Y$ be complex spaces and let $f: X \rightarrow Y$ be a proper holomorphic map. Then $f$ admits a unique \emph{Stein factorization}
$$
X \stackrel{\hat{f}}{\rightarrow} \hat{Y} \stackrel{g}{\rightarrow} Y
$$
through a complex space $\hat{Y}$ with the following properties:
\begin{itemize}
 \item $\hat{f}$ is a proper holomorphic surjection; $g$ is finite and holomorphic; $f = g \circ \hat{f}$;
 \item $\hat{f}_\ast(\Oo_X) = \Oo_{\hat{Y}}$; in particular all fibers of $\hat{f}$ are connected.
\end{itemize}
If $X$ is normal, then $\hat{Y}$ is normal too.
\end{thm}

Let $\rho: \pi_1(X) \rightarrow GL(n,\C)$ be a semisimple representation of the fundamental group of a compact K\"ahler manifold $X$. By the Corlette-Donaldson theorem, there is a harmonic bundle $(\Ee,D,h)$ of rank $n$ with monodromy $\rho$. 

\begin{thm}\label{thm_shafa}
 Let $X$ be a connected compact K\"ahler manifold and let $\rho: \pi_1(X) \rightarrow GL(n,\C)$ be a semisimple representation with discrete image. Let $f: \tilde{X} \rightarrow \Dd$ and $\rho_{tot}: \pi_1(X) \rightarrow \Lambda_\sigma G$ be the period map and total monodromy of the harmonic bundle $(\Ee,D,h)$ with monodromy $\rho$. We assume that the image of the total monodromy is without torsion. Then, the image of the map $\bar{f}: X \rightarrow \Dd/\rho_{tot}(\pi_1(X))$ is a finite-dimensional complex space. Moreover, the map $\hat{f}: X \rightarrow \hat{Y}$ in its Stein factorization 
$$X \stackrel{\hat{f}}{\rightarrow} \hat{Y} \rightarrow \Im(\bar{f})$$
is a Shafarevich morphism for $(X,\rho)$.
\end{thm}

\begin{proof}
Since $\rho(\pi_1(X))$ is discrete in $G$, $\rho_{tot}(\pi_1(X))$ is discrete in $\Lambda_\sigma G$. The group $\rho_{tot}(\pi_1(X))$ acts on $\Dd$ with stabilizer a discrete subgroup of $K$, which has to be finite, and thus trivial since we assume that $\rho_{tot}(\pi_1(X))$ has no torsion. Hence, the quotient $\Dd/\rho_{tot}(\pi_1(X))$ has a structure of complex Hilbert manifold. The map $\bar{f}: X \rightarrow \Dd/\rho_{tot}(\pi_1(X))$ is proper since $X$ is compact; by a generalization of Remmert proper mapping theorem, for maps whose target space is infinite-dimensional (corollary 3, page 180 in \cite{Maz}), the image of $\bar{f}$ is a finite-dimensional complex analytic subspace of $\Dd/\rho_{tot}(\pi_1(X))$. Let $X \stackrel{\hat{f}}{\rightarrow} Y \rightarrow \Im(\bar{f})$ be its Stein factorization. Then $\hat{f}$ is a proper holomorphic surjection with connected fibers and $\hat{Y}$ is a normal space. 

Let $Z$ be a connected complex manifold and let $g: Z \rightarrow X$ be a holomorphic map. If the map $\bar{f} \circ g: Z \rightarrow \hat{Y}$ is constant, then the period map of the variation of loop Hodge structures on $Z$ (obtained \emph{via} pullback from $g$) is constant. This simply means that the pullback of the harmonic bundle on $X$ is trivial on $Z$, in the sense that the harmonic metric is flat. Hence, the image of $\pi_1(Z)$ by $\rho \circ g_\ast$ is contained in a unitary group. Since it is also discrete, it has to be finite. Conversely, it the image of $\pi_1(Z)$ by $\rho \circ g_\ast$ is finite, it is contained in a unitary group. Hence, the flat bundle associated to the representation $\pi_1(Z) \rightarrow GL(n,\C)$ admits a flat metric, which has to be \emph{the} harmonic metric. This shows that the period map $\bar{f} \circ g: Z \rightarrow \Dd/\rho(\pi_1(X))$ is constant. This factorizes through the map $\hat{f} \circ g: Z \rightarrow \hat{Y}$, which has finite image since the map 
$\hat{Y} \rightarrow \Dd/\rho(\pi_1(X))$ is finite. Since $Z$ is connected, the map $\hat{f} \circ g$ is in fact constant. This concludes the proof. 
\end{proof}

\begin{rem}
 Following for instance proposition 3.5 in \cite{Eys}, one can try to deduce properties of $\Sh_\rho(X)$. In the case where $\rho$ is the monodromy of a variation of Hodge structures, the period map takes values in a classical period domain $\Dd_c$. The invariant metric $g$ on $\Dd_c$ is K\"ahler in the horizontal directions. Moreover, the holomorphic bisectional curvature is nonpositive in the horizontal directions and the holomorphic sectional curvature is negative in the horizontal directions. This shows that (some desingularization of) $\Sh_\rho(X)$ has ample canonical bundle since its curvature form (with respect to the K\"ahler metric $g_{|\Sh_\rho(X)}$) is positive.

On the loop period domain, we still have a $\Lambda_\sigma G$-invariant metric. It is also K\"ahler in the horizontal directions and its holomorphic bisectional curvature is still nonpositive in the horizontal directions. However, the holomorphic sectional curvature can vanish in the horizontal directions and we cannot conclude as before. One should notice that this problem was already encountered in \cite{Mok}.
\end{rem}

% ================================== BIBLIOGRAPHIE =============================

%% Choix du style 
%% En français
%\bibliographystyle{alpha-fr} % style alphabétique en français
%% En anglais
\bibliographystyle{alpha} % style alphabétique en anglais
%\bibliographystyle{plain} % style numéroté en anglais
%% Il y a plein d'autres possibilités
%% Fabrication de la biblio
\bibliography{bib_harm} % pour afficher la biblio
% utilise le fichier bibliothese.bib

% Remarque : l'ajout de la biblio à la table des matières se fait par le
% paquet tocbibind (car la commande addcontentsline ne fabrique pas le bon
% numéro de page)

% =============================== INDEX DES NOTATIONS ==========================

%\cleardoublepage % pour forcer l'index à apparaître sur une page impaire
%\renewcommand{\indexname}{Index des notations} % pour changer le nom de l'index
%\printindex % pour afficher l'index

\end{document}